\documentclass[reqno,a4paper]{amsart}
\usepackage[utf8]{inputenc}

\title[
On {B}ayesian data assimilation for ill-posed {PDE}s
]{
On Bayesian data assimilation for PDEs with ill-posed forward problems
}
\author{
S.~Lanthaler, 
S.~Mishra,
F.~Weber
}
\date{\today}

\usepackage{amsthm,amssymb}
\usepackage{mathtools,thmtools}
\usepackage{bm,esint}
\usepackage{color}
\usepackage{float,graphicx,subcaption}
\usepackage{hyperref}
\usepackage{cleveref}
\usepackage{cancel}
\usepackage{tikz,tikz-cd}
\usepackage[shortlabels]{enumitem}
\usepackage{mathrsfs}  
\usepackage{todonotes}
\usepackage{euscript}
\usepackage{cite} % references in order; not [1,6,3,2] --> [1-3,6]

\makeatletter
\def\namedlabel#1#2{\begingroup
    #2%
    \def\@currentlabel{#2}%
    \phantomsection\label{#1}\endgroup
}
\makeatother

\renewcommand{\tilde}{\widetilde}

\newcommand{\mup}{\mu_{\mathrm{prior}}}
\newcommand{\mun}{\mu_{\mathrm{noise}}}

\DeclareMathOperator*{\esssup}{ess\,sup}
\DeclareMathOperator*{\essinf}{ess\,inf}

\renewcommand{\i}{\mathbf{i}}
\newcommand{\e}{\mathbf{e}}

\newcommand{\cl}{{\mathrm{cl}}}

\newcommand{\proj}{\mathrm{Proj}}

\newcommand{\explain}[2]{\overset{\mathclap{\underset{\downarrow}{#2}}}{#1}}

\newcommand{\loc}{{\mathrm{loc}}}

\newcommand{\slot}{{\,\cdot\,}}

\newcommand{\weaklyto}{{\rightharpoonup}}

\newcommand{\T}{\mathbb{T}}
\newcommand{\R}{\mathbb{R}}
\newcommand{\E}{\mathbb{E}}

\newcommand{\Prob}{\mathbb{P}}
\renewcommand{\S}{\mathscr{S}}
\renewcommand{\div}{{\mathrm{div}}}

\newcommand{\Lip}{\mathrm{Lip}}

\newcommand{\diff}{\EuScript{D}}
\newcommand{\op}{\cS^\dagger}%{\EuScript{G}}
\newcommand{\bv}{{\mathbf{v}}}

\newcommand{\bd}{\EuScript{B}}
\renewcommand{\P}{{\mathcal{P}}}
\newcommand{\cP}{{\mathcal{P}}}

\newcommand{\cF}{\mathcal{F}}
\newcommand{\G}{\mathcal{G}}
\newcommand{\cG}{\mathcal{G}}
\renewcommand{\L}{\mathcal{L}}
\newcommand{\cL}{\mathcal{L}}
\newcommand{\tL}{\tilde{\cL}}
\newcommand{\Normal}{\mathcal{N}}

\newcommand{\cZ}{\mathcal{Z}}
\newcommand{\cS}{\mathcal{S}}

\newcommand{\define}{\textbf}

\newcommand{\set}[2]{{\left\{ #1 \,\middle|\, #2 \right\}}}

\declaretheoremstyle[
  headfont=\normalfont\bfseries\itshape,
  numbered=unless unique,
  bodyfont=\normalfont,
  spaceabove=1em plus 0.75em minus 0.25em,
  spacebelow=1em plus 0.75em minus 0.25em,
  qed={},
]{deflt}

\theoremstyle{deflt}
\newtheorem{theorem}{Theorem}[section]
\newtheorem{assumption}[theorem]{Assumption}
\newtheorem{example}[theorem]{Example}
\newtheorem{remark}[theorem]{Remark}

\newtheorem{definition}[theorem]{Definition}
\newtheorem{lemma}[theorem]{Lemma}
\newtheorem{proposition}[theorem]{Proposition}

\newtheorem{result}[theorem]{Result}

\numberwithin{equation}{section}
\numberwithin{theorem}{section}

\begin{document}

\maketitle

\begin{abstract}
We study Bayesian data assimilation (filtering) for time-evolution PDEs, for which the underlying forward problem may be very unstable or ill-posed. Such PDEs, which include the Navier-Stokes equations of fluid dynamics, are characterized by a high sensitivity of solutions to perturbations of the initial data, a lack of rigorous global well-posedness results as well as possible non-convergence of numerical approximations. Under very mild and readily verifiable general hypotheses on the forward solution operator of such PDEs, we prove that the posterior measure expressing the solution of the Bayesian filtering problem is stable with respect to perturbations of the noisy measurements, and we provide quantitative estimates on the convergence of approximate Bayesian filtering distributions computed from numerical approximations. For the Navier-Stokes equations, our results imply uniform stability of the filtering problem even at arbitrarily small viscosity, when the underlying forward problem may become ill-posed, as well as the compactness of numerical approximants in a suitable metric on time-parametrized probability measures.
\end{abstract}

\section{Introduction}

Partial differential equations (PDEs) are ubiquitous as mathematical models in the sciences and engineering. A time-dependent PDE takes the following generic form, \begin{equation}
    \label{eq:pde}
    \begin{aligned}
    \partial_t u + \diff\left(f,u,\nabla_x u,\nabla^2_x u,\cdots \right) &= 0, \quad \forall x \in D, t \in (0,T) \\
    \bd u &= \hat{u}, \quad \forall x \in \partial D, t \in (0,T), \\
    u(x,0) &= \bar{u}, \quad \forall x \in D.
    \end{aligned}
\end{equation}
Here, $\diff$ is a differential operator that depends on the solution $u$ and its spatial derivatives, as well as on a coefficient (source term) $f$. The PDE is supplemented with initial conditions and with boundary conditions, imposed through a boundary operator $\bd$. The inputs to the PDE are given by $\bv = \left[\bar{u},\hat{u},f \right]$, which constitute the initial data, boundary data and coefficients (source terms). These inputs are related to the solution $u$ of the PDE \eqref{eq:pde} through the so-called \emph{data-to-solution operator}, 
\begin{equation}
    \label{eq:solop}
    \op: X \mapsto Y, \quad \bv \mapsto \op(\bv) = u,
\end{equation}
with $u$ solving the PDE \eqref{eq:pde}. $X$ and $Y$ are suitable (subsets of) Banach spaces. 

Often, one is interested, not just in the solution field $\bv$ of \eqref{eq:pde}, but rather in finite-dimensional \emph{quantities of interest} or \emph{observables}, which are given in the generic form,
\begin{equation}
    \label{eq:obs}
\L^\dagger : \, X \to \R^d,
\qquad
\bv \mapsto \L^\dagger\left(\bv\right).
\end{equation}
The observable $\L^\dagger$ can be written as a composition $\L^\dagger(\bv) = \G^\dagger(\op(\bv))$, with $\G^\dagger: Y \to \R^d$ a functional (e.g. given by point evaluations or local averages).
Thus, the so-called \emph{forward problem} for the PDE \eqref{eq:pde}, is to evaluate the solution operator $\op$ or the observables $\L^\dagger$, given the inputs $\bv$.

However, it is not always possible to exactly \emph{know} the inputs $\bv$ (initial and boundary data, coefficients, sources etc). Rather in practice, one has to infer information about the inputs $\bv$, and consequently the solution $u$, from \emph{measurements} of the observables in \eqref{eq:obs}. Moreover in general, these measurements are \emph{noisy}. Thus one has to solve the so-called \emph{inverse problem} for a PDE, i.e., \emph{determine the input $\bv$ (and solution $u$) for the PDE \eqref{eq:pde}, given measurements of the form,}
\begin{align} \label{eq:meas}
y = \L^\dagger(\bv) + \eta,
\qquad \eta \sim \rho(y) \, dy,
\end{align}
with the noise sampled from a probability measure on $\R^d$, defined by its density $\rho$. 

It is well known that, in general, a \emph{deterministic formulation} of the afore-mentioned inverse problem can be \emph{ill-posed}. Although different regularization procedures have been developed over the last few decades to deal with this ill-posedness, it is now well-established that a statistical formulation of the inverse problem, based on a Bayesian framework, is very suitable in this context \cite{Tarantola2005,KaipioSomersalo2006,Stuart2010}.

Within a Bayesian formulation of the inverse problem, associated with the mapping \eqref{eq:obs} and measurements \eqref{eq:meas}, one encodes statistical information about the system (say inputs $\bv$ in \eqref{eq:pde}) in terms of a \emph{prior probability measure}. The additional information from the measurements \eqref{eq:meas} can be used to improve the prior by an application of the well-known Bayes' theorem \cite{Stuart2010}. This results in a so-called \emph{posterior} probability measure, on the inputs $\bv$, which represents the conditional probability of the underlying inputs, given the measurements \eqref{eq:meas}. Thus, the \emph{Bayesian Inverse Problem} (BIP) can be interpreted as a \emph{mapping from the measurements \eqref{eq:meas} to the posterior measure}. 

In contrast to the generic situation for deterministic inverse problems, it has been shown that the corresponding Bayesian inverse problem for PDEs is often \emph{well-posed}, i.e., the posterior measures exists, is unique and \emph{depends continuously} (in suitable metrics) on the measurements \eqref{eq:meas} \cite{Stuart2010,Latz2020,Sprungk2020}. Furthermore, Bayesian inverse problems can incorporate the deterministic formulation of regularized ill-posed inverse problems: As shown in \cite{Stuart2010}, the latter can often be viewed as the maximum a posteriori (MAP) estimator of a Bayesian inverse problem with a suitable choice of the underlying prior.

However, these remarkable well-posedness results for Bayesian inverse problems for PDEs rely on the well-posedness of the underlying \emph{forward problem}, often requiring that the mapping $\L^\dagger$ in \eqref{eq:obs} is \emph{Lipschitz continuous} in suitable metrics and converting this Lipschitz continuity into stability results for the posterior measure with respect to perturbations in the measurements, see \cite{Stuart2010} for a survey of these results and their applications to a variety of PDEs. More recently in \cite{Latz2020,Sprungk2020}, these Lipschitz continuity assumptions on the forward map $\L^\dagger$ in \eqref{eq:obs}, have been considerably relaxed. In particular, under suitable assumptions on the measurement noise $\eta$ in \eqref{eq:meas}, mere \emph{existence and measurability} of the forward map suffices for the well-posedness of the underlying Bayesian inverse problem \cite{Latz2020}.    

For models describing the \emph{temporal evolution} $t\mapsto u(t)$ of a system, the forward operator $\cS^\dagger$ can either be interpreted as a mapping from the given data to a space of time-dependent solutions $t \mapsto u(t)$, or equivalently, as a time-parametrized operator $\cS_t^\dagger: X \to X$, such that $t \mapsto u(t) = \cS^\dagger_t(\bar{u};\hat{u},f)$ describes the evolution of the system. For such systems, which include the PDE \eqref{eq:pde}, Bayesian inversion can be used to estimate the initial state $\bar{u}$, the boundary conditions $\hat{u}$, or the source term $f$, under very general conditions on the measurement operator $\cL^\dagger$. However, for many problems of practical importance, one is ultimately interested in an estimate of the \emph{underlying state} $u(t)$ at the present or a future time $t$. The resulting \emph{data assimilation} or \emph{filtering} problem thus seeks to blend measurement data with the underlying evolution model to make predictions about the future state. Besides its intrinsic interest, one motivation for studying a statistical viewpoint of data assimilation based on a Bayesian approach \cite{apte2008data} comes from the fact that many popular methods for data assimilation, such as the three-dimensional variational filter (3DVAR) \cite{courtier1998ecmwf}, the four-dimensional variational filter (4DVAR) \cite{rabier2000ecmwf} or the ensemble Kalman filter (e.g. \cite{evensen2009data} and references therein) can suitably be interpreted as arising from MAP estimators or Gaussian approximations of this Bayesian approach \cite{apte2008data}. In fact, the Bayesian formulation has been proposed as a ``gold-standard'' against which other methods can be evaluated \cite{EvaluatingDataAssimilationAlgorithms}. A mathematically rigorous introduction to data assimilation from this Bayesian perspective is presented in \cite{law2015data}, where attention is restricted to finite-dimensional models and Gaussian noise. To the best of the authors' knowledge, a systematic investigation of the well-posedness of Bayesian data assimilation for \emph{infinite-dimensional models} arising from PDEs, and the extension of the corresponding theory on Bayesian inverse problems of \cite{Stuart2010} to the data assimilation setting, has so far been outstanding.

Data assimilation is of particular importance in the context of fluid flows. For many fluid models, it is well-known that predictions of future states can depend very sensitively on small perturbations of the initial data (or boundary data, source terms, etc.) \cite{pope2001turbulent,frisch1995turbulence}. This sensitivity to small changes can render the forward evolution (effectively) \emph{ill-posed}. Prototypical examples for such ill-posed PDEs are provided by the fundamental equations of fluid dynamics, such as the incompressible Navier-Stokes or the compressible Euler equations. For the incompressible Navier-Stokes equations, there are currently no global well-posedness results in three space dimensions. Although admissible weak solutions exist \cite{Leray1934}, the uniqueness, stability and regularity of such solutions are outstanding open problems. Even for the two-dimensional Navier-Stokes equations, for which existence and uniqueness results have been obtained \cite{Ladyzhenskaya}, the known stability estimates for the forward problem exhibit a very unfavourable, exponential dependence $\sim \exp(t/\nu)$ on the viscosity\footnote{Physically, the non-dimensional quantity to consider is the Reynolds number $1/\mathrm{Re}\propto \nu$, obtained after suitable normalization of the equation. We will assume that the equations are suitably scaled, and will not distinguish between Reynolds number and the viscosity.} $\nu$, reflecting the high sensitivity to small perturbations of the initial data. As the viscosity $\nu \ll 1$ is often a very small number in applications, this can render even the two-dimensional Navier-Stokes equations so unstable, as to be effectively ill-posed. Similar remarks apply to the compressible Euler equations, which are canonical examples of \emph{hyperbolic systems of conservation laws} \cite{Dafermos2005}. In this case, there are no rigorous global-in-time well-posedness results in either two or three space dimensions, reflecting a lack of stability of the underlying forward solution operator $\cS^\dagger_t$, or of (numerical) approximations thereof $\cS^\Delta_t \approx \cS^\dagger_t$.
Due to the lack of stability of the forward problem for these fundamental equations of fluid dynamics, it is thus not clear to what extent the well-posedness results of \cite{Stuart2010,Latz2020,Sprungk2020} obtained for Bayesian inverse problems can be extended to this time-varying setting. Indeed, Bayesian inversion apparently only yields estimates on the \emph{initial state} in this setting, whereas data assimilation involves an additional prediction step to estimate \emph{future states}.

This lack of stability of many fluids with respect to perturbations of the initial data motivates the following question: Does Bayesian data assimilation suffer from a similar sensitivity to perturbations in the measurement data? I.e., is the well-posedness of the Bayesian data assimilation problem contingent on the well-posedness of the corresponding forward problem? The Bayesian framework has been remarkably successful in the context of weather forecasting, climate modeling and oceanography \cite{tapio}. Given that the underlying models include the incompressible Navier-Stokes and the compressible Euler equations as the core governing PDEs, how does one reconcile the empirical success of the Bayesian framework with the lack of stability of the underlying forward problem?

This dichotomy sets the stage for the current article where we investigate the well-posedness of Bayesian data assimilation for PDEs where the forward evolution operator may be ill-posed. Besides investigating the well-posedness of Bayesian data assimilation for the exact solution operator $\cS^\dagger_t$, we also consider approximations to the forward map, $\cS^{\Delta}_t \approx \cS^\dagger_t$, which may stem from numerical approximations of the underlying PDE \eqref{eq:pde}. Such approximations 
%are well-defined and 
lead to a family of \emph{approximate posteriors} for the Bayesian data assimilation problems. In this article, we will prove, under very general hypotheses, that
\begin{itemize}
    \item The Bayesian filtering problem is well-posed under mild assumptions, even if the forward problem may be \emph{ill-posed}; in particular, the mapping from measurements to posterior is uniformly Lipschitz continuous, \emph{independently} of the stability of the forward problem. 
    \item Under mild conditions on the convergence of approximate solution operators $\cS^\Delta_t \to \cS^\dagger_t$, the corresponding approximate Bayesian posteriors are consistent, in the sense that they converge in a suitable metric to the exact posterior as $\Delta \to 0$, and with the same convergence rate.
    \item Suitable families of approximate posteriors for the Navier-Stokes equations (and related equations) are compact in an appropriate metric, as $\Delta \rightarrow 0$, with uniformly continuous dependence on the measurements $\bm{y}$. This allows us to define a non-empty set of candidate solutions for the limiting Bayesian data assimilation problem, as $\Delta \to 0$, even for models for which there are no known convergence guarantees, $\cS^\Delta_t \overset{??}{\to} \cS^\dagger_t$, for the forward problem, such as the \emph{three-dimensional} Navier-Stokes equations.
    \end{itemize}
Although uniqueness of the posterior is not necessarily guaranteed with these compactness arguments, our construction could pave the way for proposing additional selection criteria on the set of approximate posteriors to recover uniqueness. 

\subsection{Organization} This work is organized as follows: In section \ref{sec:notation}, we introduce the precise mathematical setting. To this end, we first formalize Bayesian data assimilation in the infinite-dimensional setting considered in the present work, and provide a formal definition of well-posedness and consistency, following similar considerations as for the Bayesian inverse problem (BIP) in \cite{Stuart2010,Latz2020,Sprungk2020}. We also briefly review key elements of the well-posedness theory for the Navier-Stokes equations in section \ref{sec:NS}, which serves as our main prototypical model, motivating the present work. In section \ref{sec:main}, we point out the precise connection between Bayesian inversion and Bayesian data assimilation (filtering), before stating our main results regarding the well-posedness, consistency and uniform stability of Bayesian filtering (cp. section \ref{sec:wpBDA}). The technical details of the mathematical derivation of these main results are collected in section \ref{sec:derivation}, where we also comment on related results for hyperbolic conservation laws. Conclusions are provided in section \ref{sec:discussion}. Some mathematical background is summarized in the appendix.

\section{Mathematical setting and notation}
\label{sec:notation}

In the present section, we introduce notation that is employed throughout this work, and provide definitions for the Bayesian data assimilation problems of interest. Besides setting the background for our main results, summarized in the subsequent section \ref{sec:main}, we will also review some key results on the well-posedness and numerical approximation of the Navier-Stokes equations which have largely motivated the present work on the well-posedness of the corresponding Bayesian data assimilation problem, and the convergence of approximate posteriors obtained by discretization. 

Throughout this work, we follow the convention that constants $C$ appearing in estimates may change their value from line to line. The dependency of the constant $C$ on the given data (e.g. parameters $\alpha,\beta,\gamma$) should usually be clear from the context and will be indicated by writing $C = C(\alpha,\beta,\gamma)$.

\subsection{Bayesian data assimilation}
\label{sec:BDA}

Data assimilation (DA) seeks to provide an estimate for the underlying state of a system, by combining available measurements with a model of the system. The temporal evolution of the system's state can often be described by a forward solution operator $\cS^\dagger_t: X \to X$ depending on time $t\in [0,T]$ and mapping the initial data $\bar{u} \in X$ to the solution at time $t$. Here, we assume $X$ to be a Banach space, equipped with a norm $\Vert \slot \Vert_X$. The evolution of the system starting from initial state $\bar{u}$ is thus given by $t \mapsto \cS^\dagger_t(\bar{u})$. In view of the application to ill-posed problems, we will make essentially no assumptions on the regularity of $\cS^\dagger_t$; in fact, unless otherwise stated, we will merely assume that:
\begin{itemize}
\item[\namedlabel{ass:sol1}{($\cS$.1)}] 
The solution operator defines a \emph{Borel measurable} mapping 
\begin{align} \label{eq:sol1}
\cS^\dagger: [0,T]\times X \to X, \quad (t,\bar{u}) \mapsto \cS^\dagger_t(\bar{u}),
\tag{$\cS$.1}
\end{align}
with $t\mapsto \cS_t(\bar{u})$ and $\bar{u}\mapsto \cS_t(\bar{u})$ measurable for all $t\in [0,T]$, $\bar{u}\in X$.
\item[\namedlabel{ass:sol2}{($\cS$.2)}] 
There exists a constant $B_\cS>0$, such that 
\begin{align} \label{eq:sol2}
\Vert \cS^\dagger_t(\bar{u}) \Vert_{X} \le B_{\cS} \Vert \bar{u} \Vert_X,
\tag{$\cS$.2}
\end{align}
for all $\bar{u},\bar{u}'\in X$, and $t \in [0,T]$.
\end{itemize}

The Bayesian data assimilation problem can then be stated as follows: Given a prior probability measure $\mup \in \cP(X)$, we consider the initial state as a random variable $\bar{u} \sim \mup$. Given a time interval $[0,T]$ and a sequence $0=t_0 < t_1 < \dots < t_N$, we assume that noisy measurements $y_1,\dots, y_N$ are made, where $y_j$ depends only on the underlying state during the time interval $[t_{j-1},t_j]$ and is of the form:
\begin{align}
\label{eq:meas1}
y_j = \cG_j(\cS^\dagger(\bar{u})) + \eta_j.
\end{align}
These measurements \eqref{eq:meas1} are defined in terms of certain measurement functionals $\cG_j$ and random variables $\eta_j \sim \mun$ modeling (additive) measurement noise, both of which are further specified next.

\begin{figure}[H]
\centering
\def\svgwidth{.5\columnwidth}
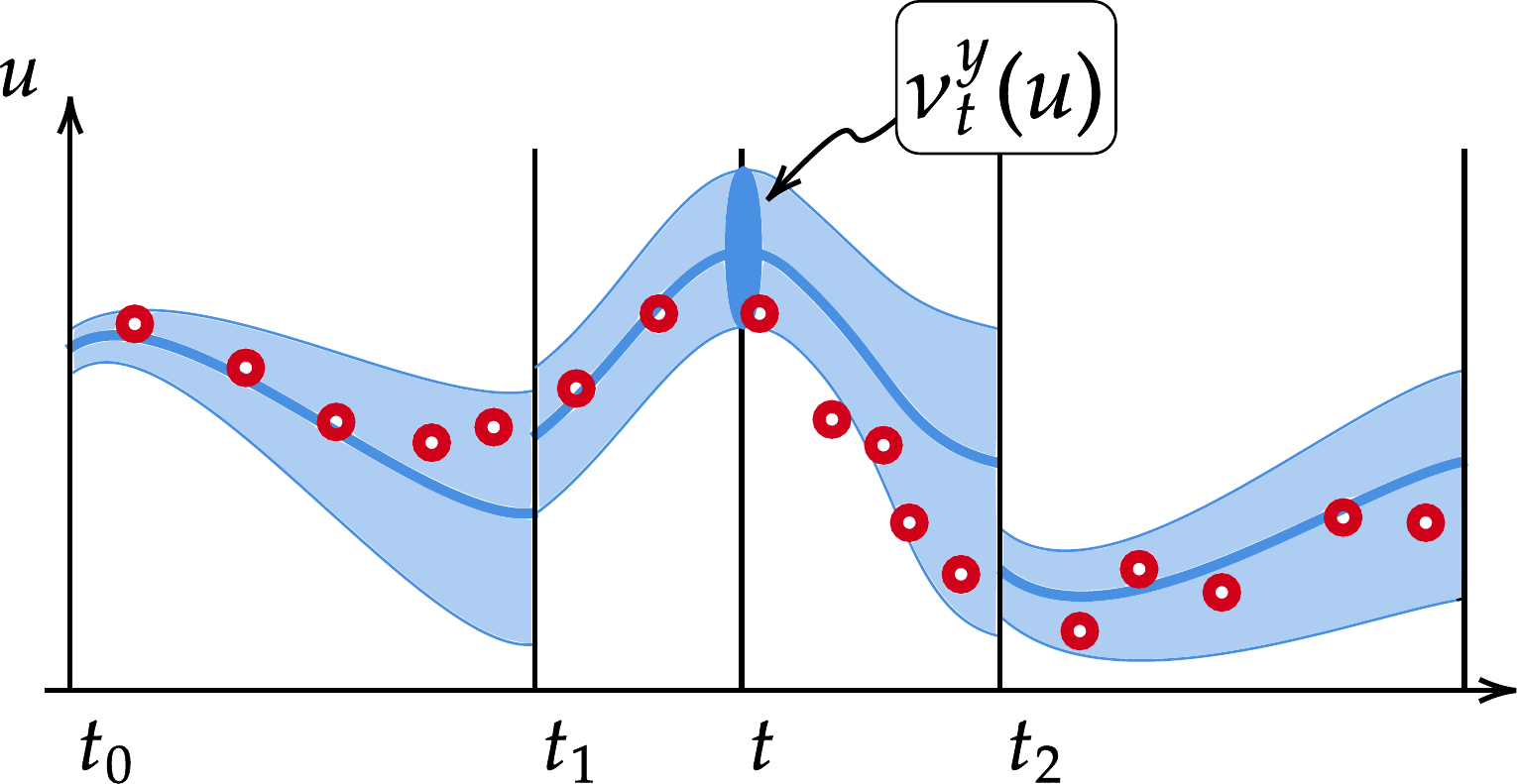

\caption{Schematic illustration of the data assimilation problem: Measurements (red circles) are used at times $t = t_0$, $t_1$, ..., to periodically update the posterior measure $\nu^{\bm{y}}_t$ (indicated by its confidence interval in blue), combining all available information from the deterministic evolution and noisy measurements.}

\label{fig:filtering}
\end{figure}

\subsubsection{Assumptions on observables}
The (potentially non-linear) functionals $\cG_j: L^1([0,T];X) \to \R^d$, $j=1,\dots, N$, will be referred to as \emph{observables}, and are assumed to depend only on the values $\cS_t^\dagger(\bar{u})$ for $t\in [t_{j-1},t_j]$. We will say that the observables $\cG_j$ are Lipschitz continuous, if there exists $L_\cG >0$, such that
\begin{align}
\label{eq:Lj}
| \cG_j(u) - \cG_j(u') | &\le L_\cG \int_{t_{j-1}}^{t_j} \Vert u - u' \Vert_{X} \, dt,
\quad
\forall u,u' \in L^1([0,T];X),
\end{align}
for all $j=1,\dots, N$. Here, we recall that the space $L^1([0,T];X)$ consists of all measurable mappings $u: [0,T] \to X$, such that $\int_0^T \Vert u(t) \Vert_X \, dt < \infty$.
We note that by \ref{ass:sol1}--\ref{ass:sol2}, we have $t\mapsto \cS^\dagger_t(\bar{u}) \in L^1([0,T];X)$ for any $\bar{u}\in X$, and hence, we have a well-defined mapping $\cS^\dagger: X \to L^1([0,T];X)$, $\bar{u} \mapsto \cS^\dagger(\bar{u})$. In particular, the composition $\cG_j(\cS^\dagger(\bar{u}))$ in \eqref{eq:meas1} is well-defined.

\begin{example}[Eulerian observables]
If $X = L^2(D;\R^n)$ consists of square-integrable functions on a bounded domain $D \subset \R^m$, then a specific class of (Eulerian) observables are functionals $\cG_j: L^1([0,T];X) \to \R^d$ of the following form:
\[
\cG_j(u) = \int_{t_{j-1}}^{t_j} \int_{D} \phi(x,t) g(u(x,t)) \, dx \, dt,
\]
where $\phi \in L^\infty(D\times [t_{j-1},t_j])$, and $g: \R^n \to \R^d$ is a Lipschitz continuous function.
\end{example}

\subsubsection{Assumptions on measurement noise}

The measurement noise is modeled by random variables $\eta_1,\dots, \eta_N \sim \mun$ which, for simplicity, we assume to be iid and independent of $\bar{u}\sim \mup$. Fix a symmetric, positive definite matrix $\Gamma \in \R^{d\times d}$, and denote by $|\slot |_\Gamma$ the corresponding norm on $\R^d$ given by 
\begin{align} \label{eq:Gnorm}
\left|
y
\right|_{\Gamma}
=
\sqrt{\langle y,y\rangle_\Gamma},
\quad
\langle y, y' \rangle_\Gamma
=
\langle y, \Gamma^{-1} y'\rangle,
\end{align}
with $\langle \slot, \slot \rangle$ the standard Euclidean inner product on $\R^d$.
We assume that the noise $\eta_j \sim \mun = \rho(y) \, dy$ in \eqref{eq:meas1} possesses a distribution that is absolutely continuous with respect to Lebesgue measure $dy$ on $\R^d$ with probability density $\rho(y) > 0$, satisfying the following assumptions:
\begin{itemize}
\item[\namedlabel{ass:noise1}{(N.1)}] \emph{Regularity:} $y\mapsto \rho(y)$ is Lipschitz continuous with respect to $|\slot|_\Gamma$,\footnote{Although all norms on the finite-dimensional space $\R^d$ are equivalent, measurement noise such as Gaussian noise is naturally associated with the norm $|\slot|_\Gamma$ induced by the covariance matrix $\Gamma$.}, i.e. there exists $L_\rho > 0$, such that 
\begin{align} \label{eq:noise1}
|\rho(y) - \rho(y')| \le L_\rho |y-y'|_\Gamma, \quad \forall \, y,y'\in \R^d.
\tag{N.1}
\end{align}
\item[\namedlabel{ass:noise2}{(N.2)}] \emph{Boundedness:} $y \mapsto \rho(y)$ is bounded from above, i.e. there exists $C > 0$, such that 
\begin{align}
\sup_{y\in \R^d} \rho(y) \le C.
\tag{N.2}
\end{align}
\item[\namedlabel{ass:noise3}{(N.3)}] \emph{Tail-condition:} there exists a constant $C>0$, such that
\begin{align} \label{eq:tail}
\rho(y) \ge \frac{\exp\left({-\frac12|y|_\Gamma^2}\right)}{C}, \quad \forall \, y\in \R^d.
\tag{N.3}
\end{align}
\end{itemize}

\begin{remark}
Note that if, instead of \eqref{eq:tail}, $\rho(y)$ satisfies a tail-condition of the form $\rho(y) \ge \exp(-C|y|_\Gamma^2)/C$, then upon simply rescaling $\tilde{\Gamma} := \sqrt{2/C}\, \Gamma$, we have $\rho(y) \ge \exp(-\frac12 |y|_{\tilde{\Gamma}}^2)/C$. Hence $\rho(y)$ satisfies assumptions \ref{ass:noise1}--\ref{ass:noise3} with a rescaled matrix $\Gamma \to \tilde{\Gamma}$ in this case. Therefore, the precise constant $\frac12$ in the tail-condition \eqref{eq:tail} can be assumed without loss of generality. The factor of $1/2$ turns out to be particularly convenient.
\end{remark}

Assumptions \ref{ass:noise1}--\ref{ass:noise3} are clearly fulfilled for normally distributed measurement noise $\eta_j \sim \mathcal{N}(0,\Gamma)$. This is the main application we have in mind. However, it is worth pointing out that the assumption is satisfied for a much wider class of measurement noise: In particular, since the tail-condition requires only a lower bound, our results apply to situations in which one encounters noise \emph{with a heavy tail}.

\subsubsection{Posterior/conditional probability}
Given a time $t \ge 0$ and a given subset of measurements $y_{1:k} = (y_1,\dots, y_k)$, we are interested in the conditional probability
\begin{align} 
\label{eq:DAsol}
\nu^{y_{1:k}}_{t}(\slot)
:= 
\Prob\left[
\cS_{t}^\dagger(\bar{u}) \in \slot 
\,\Big|\,
\cG_j(\cS^\dagger(\bar{u})) + \eta_j = y_j, \, \forall j=1,\dots, k
\right],
\end{align}
providing a Bayesian estimate of the underlying state $u(t) = \cS^\dagger_{t}(\bar{u})$ at time $t\ge 0$ given the prior distribution $\mup(d\bar{u})$ at $t=0$ and the measurements $y_1,\dots, y_k$. If all available measurements at past times $t_j \le t$ are taken into account, this estimate is referred to as the \emph{filtering distribution}; if the estimate also takes into account measurements obtained at times $t_j \ge t$, i.e. the state is estimated in hindsight, the posterior is referred to as the \emph{smoothing distribution}. In the case of filtering, the set of available measurements $y_{1:k}$ will itself vary with time $t$, i.e. $k=k(t)$. For concreteness, we will mostly focus on the filtering problem in the following; given all measurements $\bm{y} = (y_1,\dots, y_N)$ over a time interval $[0,T]$, the filtering distribution 
\begin{align}
\label{eq:filtsol}
\nu^{\bm{y}}_t(\slot) = \Prob[\cS^\dagger_{t}(\bar{u})\in \slot \,|\, y_{j} \,\text{with } t_j \le t],
\end{align}
 provides the best-estimate at time $t$ given \emph{only the past measurements}; the filtering distribution can be written in terms of the conditional probabilities \eqref{eq:DAsol}:
\begin{align} \label{eq:filt}
\nu^{\bm{y}}_t
:=
\begin{cases}
\nu^{\emptyset}_t, & t\in [0,t_1), \\
\nu^{y_1}_t, & t\in [t_1,t_2), \\
\vdots & \\
\nu^{y_{1:(N-1)}}, & t\in [t_{N-1},t_N), \\
\nu^{y_{1:N}}, & t\ge t_N.
\end{cases}
\end{align}
Here we have formally defined $\nu^{\emptyset}_{t} := \cS^\dagger_{t,\#} \mup$, corresponding to the best prediction in the absence of any measurements. The filtering distribution \eqref{eq:filt} thus defines a mapping from measurements $\bm{y} = (y_1,\dots, y_N) \in \R^{d\times N}$ to time-parametrized probability measures $\nu^{\bm{y}}_t$.

\subsubsection{Definition of well-posedness}
We are interested in the well-posedness of the filtering problem, as defined next: 

\begin{definition}[Well-posedness of Bayesian-DA]
\label{def:wellposed}
Given a forward operator $\cS^\dagger_t$, a prior $\mup$, a noise distribution $\mun$, and a space $(\bm{\Sigma},d_T)$ of time-parametrized probability measures $t\mapsto \nu_t$, we say that the Bayesian-DA problem is \define{well-posed}, provided that the following properties are satisfied:
\begin{enumerate}
\item \emph{Existence:} For any $\bm{y}\in \R^{d\times N}$, the posterior filtering distribution \eqref{eq:DAsol} exists in $\bm{\Sigma}$,
\item \emph{Uniqueness:} the filtering distribution is unique, 
\item \emph{Stability:} The measurement-to-posterior mapping 
\[
\R^{d\times N} \to \bm{\Sigma}, 
\quad
\bm{y} \mapsto \nu^{\bm{y}}_t,
\]
is locally Lipschitz continuous wrt. $d_T$, i.e. for any $R>0$, there exists $C(R)>0$, such that 
\[
d_T(\nu^{\bm{y}}_t,\nu^{\bm{y}'}_t) \le C |\bm{y}-\bm{y}'|_{\Gamma},
\]
for all $\bm{y},\bm{y}'\in \R^{d\times N}$, such that $|\bm{y}|_{\Gamma},|\bm{y}'|_{\Gamma} \le R$, and where we define $|\bm{y}|_{\Gamma} := \sqrt{\sum_{j=1}^N |y_j|_{\Gamma}^2}$.
\end{enumerate}
\end{definition}

One possible choice for the metric space $(\bm{\Sigma},d_T)$ will be discussed below (cp. Section \ref{sec:timep}). At this point, we would like to point out that the above definition is a direct analogue of the corresponding definition of well-posedness for the Bayesian inverse problem \cite[Def. 2.7, Def. 3.4]{Latz2020}, as well as the notion of well-posedness for the forward problem:

\begin{remark}[Well-posedness of the forward problem]
If $\cS^\dagger_t: X \to X$ is the solution operator associated with a time-evolution PDE on a time interval $[0,T]$, then the well-posedness of the forward problem is usually defined as the existence, uniqueness and stability of $\cS^\dagger_t$ on $X$, where the stability requires $\cS^\dagger_t$ to be \emph{continuous} as a mapping $\bar{u} \to \cS^\dagger_t(\bar{u})$ for any $t\in [0,T]$. In fact, $\cS^\dagger_t$ is often required to be \emph{Lipschitz continuous},  i.e. there exists a constant $L_\cS\ge 0$, such that 
\[
\Vert \cS^\dagger_t(\bar{u}) - \cS^\dagger_t(\bar{u}') \Vert_X
\le 
L_\cS \Vert \bar{u} - \bar{u}' \Vert_X,
\]
for any $\bar{u},\bar{u}' \in X$. Thus, the well-posedness of the forward problem is reflected in the regularity of $\cS^\dagger_t$. In the present work, we will study the well-posedness of the associated Bayesian filtering problem \eqref{eq:filtsol} \emph{in the absence of such regularity}, thus formally allowing for $L_\cS = \infty$.
\end{remark}

\subsubsection{Numerical discretization and consistency}
As the true forward operator $\cS^\dagger_t$ is usually not computable in practice, one often needs to replace $\cS^\dagger_t$ by a numerical approximation $\cS^\Delta_t \approx \cS^\dagger_t$, depending on a parameter $\Delta > 0$. The parameter $\Delta > 0$ may reflect the grid size in a numerical discretization, or may represent more general modeling errors; in the following, we will usually refer to $\Delta$ as the ``grid size'', and will focus on errors due to discretization of a given PDE.
%, thus assuming the PDE to describe the relevant physical processes of the underlying system. 
Upon discretization, the exact posterior \eqref{eq:DAsol} is replaced by the following conditional probability:
\begin{align}
\label{eq:DAsol1}
\nu^{\Delta,y_{1:k}}_{t}(\slot)
:= 
\Prob\left[
\cS_{t}^\Delta(\bar{u}) \in \slot 
\,\Big|\,
\cG_j(\cS^\Delta(\bar{u})) + \eta_j = y_j, \, \forall j=1,\dots, k
\right].
\end{align}
The corresponding filtering distribution $\nu^{\Delta,\bm{y}}_t$ is defined as in \eqref{eq:filt}, but with $\nu^{\Delta,y_{1:k}}_t$ replacing $\nu^{y_{1:k}}_t$. Given such a discretization, a fundamental question concerns the \emph{consistency} of the approximate posteriors $\nu^{\Delta,\bm{y}}_t$ with the limiting posterior $\nu^{\bm{y}}_t$:

\begin{definition}[Consistency]
\label{def:consistency}
The sequence of approximate posteriors $\nu^{\Delta,\bm{y}}_t$ is \emph{consistent} with the limiting posterior $\nu^{\bm{y}}_t$, with respect to a space of time-parametrized probability measures $(\bm{\Sigma},d_T)$, if $\nu^{\Delta,\bm{y}}_t\to \nu^{\bm{y}}_t$ converges locally uniformly in $\bm{y} = (y_1,\dots, y_N)\in \R^{d \times N}$; i.e., if for any $R>0$, we have 
\[
\lim_{\Delta\to 0} \sup_{|\bm{y}|_\Gamma \le R} d_T(\nu^{\Delta,\bm{y}}_t, \nu^{\bm{y}}_t) = 0.
\]
\end{definition}

\subsection{Time-parametrized probability measures}
\label{sec:timep}

Given our definition of well-posedness and consistency of the Bayesian data assimilation problem, and the solution \eqref{eq:filt} of the filtering problem, we need to define a suitable space $(\bm{\Sigma},d_T)$ of time-parametrized probability measures, $t\mapsto \nu_t$. To this end, we follow \cite{LMP1}, and introduce the following space $L^1_t(\cP) = L^1_t([0,T];\cP_1(X))$:
\begin{definition}
Let $X$ be a separable Banach space with norm $\Vert \slot \Vert_X$, and let $\cP_1(X)$ denote the set of Borel probability measures $\mu$ on $X$, with finite first-moment $\int_X \Vert u \Vert_X \, \mu(du)< \infty$. We recall that $\cP_1(X)$ is metrized by the $1$-Wasserstein metric $W_1(\slot,\slot)$ (cp. Section \ref{sec:Wasserstein} for definitions). Given a time $T>0$, we define $L^1_t([0,T];\cP_1(X))$ to be the set of weak-$\ast$ measurable\footnote{For any $\Phi\in C_b(X)$, the mapping $t\mapsto \int_{X} \Phi(u) \, \nu_t(du)$ is measurable.} mappings $[0,T] \to \cP_1(X)$, $t\mapsto \nu_t$, such that
\[
\int_0^T \int_X \Vert u \Vert_{X} \, \nu_t(du) \, dt < \infty, 
\]
and we introduce the following metric on $L^1_t([0;T];\cP_1(X))$:
\[
d_T(\nu_t, \nu'_t) 
:=
\int_0^T W_1(\nu_t,\nu'_t) \, dt, 
\quad 
\forall \, \nu_t,\nu_t' \in L^1_t([0;T];\cP_1(X)).
\]
\end{definition}
We shall usually employ the simpler notation $L^1_t(\cP) = L^1_t([0,T];\cP_1(X))$, when the temporal domain $[0,T]$ and the underlying Banach space $X$ are clear from the context. Following \cite[Proposition 2.1]{LMP1}, we also recall
\begin{proposition}
Let $X$ be a separable Banach space. Then $L^1_t(\cP) = L^1_t([0,T];X)$ is a complete metric space under the norm $d_T(\slot,\slot) = \int_0^T W_1(\slot,\slot) \, dt$.
\end{proposition}

The motivation for considering this particular metric on time-parametrized probability measures is two-fold: Firstly, this metric and closely related quantities have been shown to be relevant empirically as well as analytically for the convergence of numerical approximations to so-called ``statistical solutions'', for several fundamental equations of fluid dynamics including the incompressible Navier-Stokes \cite{bansal2021numerical}, incompressible Euler \cite{LMP1,LMP2} as well as the compressible Euler equations \cite{FLMW1}. Secondly, metrics other than the Wasserstein $W_1$-metric, such as the Hellinger and total variation distances or the Kullback-Leibler divergence, which have been considered in the context of Bayesian inverse problems \cite{Latz2020,Sprungk2020}, may be less suitable in the context of Bayesian data assimilation, since these latter distances require absolute continuity of the involved measures. While this requirement of absolute continuity is often not an issue for Bayesian inverse problems \cite{Stuart2010,Latz2020,Sprungk2020}, the filtering distributions $\nu^{\bm{y}}_t$, $\nu^{\Delta,\bm{y}}_t$ considered in the present work are generally singular with respect to each other (due to the additional prediction step). Hence, we focus on distances which allow for \emph{disjoint supports} of the underlying measures, such as the Wasserstein distance.

\subsection{Navier-Stokes equations}
\label{sec:NS}

To illustrate ill-posed forward problems arising in fluid mechanics, we next review some elements of the stability theory for the incompressible Navier-Stokes equations. The Navier-Stokes equations are here viewed as a prototypical model of fluid flows, given by the following system of PDEs:
\begin{gather} \label{eq:NS}
\left\{
\begin{aligned}
\partial_t u + \div\left(u\otimes u\right) + \nabla p &= \nu \Delta u, \\
\div(u) = 0, \quad 
u(\slot ,0) &= \overline{u}.
\end{aligned}
\right.
\end{gather}
These equations describe the evolution of the flow vector field $u: D\times [0,T] \to \R^n$, $u = (u_1,\dots, u_n)$, of a fluid in $n$ dimensions. The parameter $\nu> 0$ denotes the viscosity of the fluid. The divergence term $\div(u\otimes u)$ has components $[\div(u\times u)]_i = \sum_{j=1}^n \partial_j (u_ju_i)$ ($i=1,\dots, n$), $\nabla p = (\partial_1 p, \dots, \partial_n p)$ is the gradient of the pressure $p$, and $\Delta u = \sum_{j=1}^n \partial_j^2 u$ on the right-hand side denotes the Laplacian applied to $u$. For simplicity we shall focus on the case of periodic boundary conditions. 

\begin{remark}[Setting for Navier-Stokes equations]
For $n\in \{2,3\}$, we consider initial data $\bar{u} \in L^2_x := L^2(\T^n;\R^n)$, consisting of $2\pi$-periodic $L^2$-integrable vector fields defined on the periodic torus $\T^n \simeq [0,2\pi]^n$ in $n$ dimensions. In addition, any initial data is required to be divergence-free, $\div(\bar{u})=0$. For such initial data, we seek weak solutions $u\in L^\infty([0,T];L^2_x)$ of \eqref{eq:NS}. We note that physically, the quantity $\frac12 \Vert u(t) \Vert_{L^2_x}^2$ corresponds to the kinetic energy of the underlying fluid, and hence the requirement that $\esssup_{t\in [0,T]} \Vert u(t) \Vert_{L^2_x} < \infty$ is natural.
\end{remark}

\subsubsection{Theoretical results}
\label{sec:NStheory}

 It is well-known that for very small values of the viscosity $\nu \ll 1$, solutions of the Navier-Stokes equations can exhibit turbulent behaviour, characterized by a high sensitivity to perturbations to the initial data and involving dynamics across a wide range of spatial and temporal scales \cite{pope2001turbulent,frisch1995turbulence}. This empirically observed turbulent behaviour at small viscosity is mathematically reflected by a strong $\nu$-dependence in the available a priori stability results for the solution operator $\cS^\dagger_t$ of the system \eqref{eq:NS}. This is summarized in the following well-known theorem for the two-dimensional case:
\begin{theorem}[{Stability of N-S in 2d, see \cite[p.~170, Chap.~6, Thm.~11]{Ladyzhenskaya}}]
\label{thm:NS2d}
 Let $\bar{u}\in L^2_x$ be initial data for the incompressible Navier-Stokes equations \eqref{eq:NS} for $n=2$. There exists a unique solution $u(t) = \cS^\dagger_t(\bar{u})$ of the Navier-Stokes equations with initial data $\bar{u}$. Furthermore, for any initial data $\bar{u},\bar{u}'\in L^2_x$, we have
 \[
 	\Vert \cS^\dagger_t(\bar{u}) - \cS^\dagger_t(\bar{u}') \Vert_{L^2_x}
 	\le 
 	\Vert \bar{u} - \bar{u}' \Vert_{L^2_x} \exp\left(
 	\frac{2}{\nu} \int_0^t \Vert \nabla u(\tau) \Vert_{L^2}^2 \, d\tau
 	\right).
 \]
\end{theorem}

Thus, even if the solution $u(t) = \cS^\dagger_t(\bar{u})$ is assumed to be Lipschitz continuous, Theorem \ref{thm:NS2d} provides at best a stability estimate of the form $\Vert \cS^\dagger_t(\bar{u}) - \cS^\dagger_t(\bar{u}') \Vert_{L^2_x} \le e^{Ct/\nu} \Vert \bar{u} - \bar{u}' \Vert_{L^2_x}$, which exhibits an \emph{exponential dependence on $1/\nu$}. 

In the three-dimensional case, a global existence and stability result such as Theorem \ref{thm:NS2d} remains unknown; it is well-known that solutions exist \emph{locally} in time and when starting from sufficiently regular initial data. For general initial data in $L^2_x$, it has been shown in the celebrated work of Leray \cite{Leray1934} that energy admissible solutions exist, but their uniqueness remains an open problem; in particular, \emph{there are no guarantees on the stability of a solution operator $\cS^\dagger_t: L^2_x \to L^2_x$ for the three-dimensional Navier-Stokes equations}.

\subsubsection{Numerical discretization} 
\label{sec:sv}

Popular numerical discretizations of the forward problem for the incompressible  Navier-Stokes equations, especially on periodic domains, are spectral methods \cite{GuermondPrudhomme,Ors1,Cho1,Karamanos2000a,Ghoshal}. Spectral methods are based on the following ansatz ${u}^\Delta(x,t) = \sum_{|{k}|_\infty\le N} \widehat{{u}}^\Delta_{{k}}(t) e^{i{k}\cdot{x}}$, where now and in the following we shall consistently denote $\Delta = 1/N$, and $|{k}|_\infty := \max_{i=1,\dots,d} |k_i|$. A straight-forward spectral approximation of the Navier-Stokes equations is based on a Galerkin projection onto this ansatz space:
\begin{gather} \label{eq:spectralvisc}
\left\{
\begin{aligned}
\partial_t {u}^\Delta
+\P_N \, \div({u}^\Delta\otimes {u}^\Delta) 
+ \nabla p^\Delta 
&=
\nu \Delta u^\Delta,
\\
\div({u}^\Delta) 
= 
0, 
\quad
{u}^\Delta|_{t=0} 
&=
\P_N \overline{u}. 
\end{aligned}
\right.
\end{gather}
Here $\P_N$ is the spatial Fourier projection operator, mapping an arbitrary function $f(x,t)$ onto the first $N$ Fourier modes: $\P_N f(x,t) = \sum_{|{k}|_\infty\le N} \widehat{f}_{{k}}(t) e^{i{k}\cdot{x}}$. The scheme \eqref{eq:spectralvisc} can be equivalently written in terms of a system of ODEs for the Fourier coefficients $\widehat{u}_k(t)$, $|k|_\infty \le N$. Hence, combined with a suitable (e.g. Runge-Kutta) time-stepping, \eqref{eq:spectralvisc} provides a numerical discretization of the Navier-Stokes equations. In the following proposition, we summarize some basic observations on the numerical approximations \eqref{eq:spectralvisc}:
\begin{proposition}
\label{prop:sv}
For any $\Delta = 1/N > 0$, let $\cS^\Delta_t: L^2_x \to L^2_x$ denote the solution operator associated with the numerical scheme \eqref{eq:spectralvisc}. Then for initial data $\bar{u}\in L^2_x$, the numerical solution $u^\Delta(t) = \cS^\Delta_t(\bar{u})$ satisfies:
\begin{itemize}
\item[\namedlabel{ass:sv1}{($\Delta$.1)}] \emph{Energy bound:}
%\begin{align} \label{eq:sv1}
$\Vert \cS^\Delta_t(\bar{u}) \Vert_{L^2_x} \le \Vert \bar{u} \Vert_{L^2_x},$
%\end{align}
\item[\namedlabel{ass:sv2}{($\Delta$.2)}] \emph{Coercivity:}
%\begin{align} \label{eq:sv2}
$\int_0^T \Vert \nabla u^\Delta(t) \Vert^2_{L^2_x} \, dt 
\le \nu^{-1} \Vert \bar{u} \Vert_{L^2_x}^2,
$
%\end{align}
\item[\namedlabel{ass:sv3}{($\Delta$.3)}] \emph{Weak time-regularity:}
There exist constants $C = C(\Vert \bar{u} \Vert_{L^2_x}), L > 0$, such that 
%\begin{align} \label{eq:sv3}
$\Vert u^\Delta(t) - u^\Delta(s) \Vert_{H^{-L}_x}
\le 
C |t-s|.
$
%\end{align}
In particular, we have $u^\Delta \in \Lip([0,T];H^{-L}_x)$ uniformly in $\Delta >0$.
\end{itemize}
\end{proposition}

The basic properties summarized above will form the basis for the well-posedness results of the present work. Additional control on the numerical approximations for the forward problem, especially for the 3D Navier-Stokes equations and/or rough initial data,  can be achieved by refining the scheme \eqref{eq:spectralvisc} through the addition of suitable numerical diffusion; this can provide additional control on the small scale behaviour, resulting e.g. in spectral viscosity schemes as proposed in \cite{GuermondPrudhomme,Tadmor1989,Tadmor2004,LM2019}. The basic properties of Proposition \ref{prop:sv} will, however, suffice for the purposes of the present work.

\section{Main results}
\label{sec:main}

As pointed out in the last section, for many PDEs encountered in the context of fluid dynamics (such as the Navier-Stokes equations in 3D), the current mathematical understanding does not allow to rigorously \emph{prove} the existence, uniqueness and stability of the corresponding forward problem, i.e. a unique forward solution operator $\cS^\dagger_t: X \to X$ is not known to exist, and even if it exists there may be \emph{no stability} in the sense that we could have $\Lip(\cS^\dagger_t: X \to X) = \infty$. Furthermore, even in those special cases, where the solution operator $\cS^\dagger_t$ can be shown to exist, stability estimates may exhibit a very unfavourable (exponential) dependence on small physical parameters $\nu \ll 1$, e.g. yielding $\Lip(\cS^\dagger_t) \sim e^{Ct/\nu}$ as for the Navier-Stokes equations in 2D. Such an exponential dependence on $1/\nu$ renders the forward evolution \emph{effectively ill-posed} for small values of $\nu \ll 1$. 

In view of these observations, in this section, we will summarize our results on
\begin{itemize}
\item the \textbf{well-posedness of Bayesian data assimilation} in the sense of Definition \ref{def:wellposed}, even when the forward problem may be ill-posed,
\item the \textbf{consistency of approximate posteriors} $\nu^{\Delta,\bm{y}}_t$ with the limiting distribution $\nu^{\bm{y}}_t$ in the sense of Definition \ref{def:consistency}, including \textbf{convergence rates}, when a unique solution operator $\cS^\dagger_t$ of the limiting problem exists,
\item \textbf{compactness} and \textbf{(uniform) stability properties} of the approximate filtering distributions \eqref{eq:DAsol1}, in the limit $\Delta \to 0$, even when no rigorous convergence guarantees $\cS^\Delta_t \overset{??}{\to} \cS^\dagger_t$ to a unique limiting forward solution operator are known.
\end{itemize}

In Section \ref{sec:connect-bip-da}, we first start with several remarks on the connection between inverse problems, as considered in \cite{Stuart2010,Sprungk2020,Latz2020}, and the Bayesian data assimilation (filtering) problem of the present work. We will also indicate the main mathematical difficulty encountered when considering \emph{ill-posed} forward problems, for which $\Lip(\cS^\dagger_t) = \infty$ or $\Lip(\cS^\dagger_t) \gg 1$ (cp. Proposition \ref{prop:LipW1}, below).

\subsection{Connection with Bayesian inverse problems}
\label{sec:connect-bip-da}

Closely related to the Bayesian data assimilation problem is the corresponding Bayesian inverse problem (BIP), which can be used to determine the following posterior probability on the \emph{initial data} $\bar{u}$ at time $t=0$, for $\cL_j^\dagger(\bar{u}) := \cG_j(\cS^\dagger(\bar{u}))$:
\begin{align}
\label{eq:BIPsol}
\mu^{y_{1:k}}(\slot) 
=
\Prob\left[
\bar{u} \in \slot
\,\Big|\,
\cL_j^\dagger(\bar{u}) + \eta_j = y_j, \, \forall j=1,\dots, k
\right],
\end{align}
or, upon discretization with $\cL_j^\Delta(\bar{u}) := \cG_j(\cS^\Delta(\bar{u}))$,
\begin{align}
\label{eq:BIPsol1}
\mu^{\Delta, y_{1:k}}(\slot) 
=
\Prob\left[
\bar{u} \in \slot
\,\Big|\,
\cL_j^\Delta(\bar{u}) + \eta_j = y_j, \, \forall j=1,\dots, k
\right].
\end{align}

A detailed analysis of the posterior \eqref{eq:BIPsol}, \eqref{eq:BIPsol1} has been provided for infinite-dimensional problems in \cite{Stuart2010}. An extended discussion of the well-posedness of the Bayesian inverse problem under minimal assumptions on the forward problem has been given in \cite{Latz2020,Sprungk2020}; As follows from \cite[Thm. 2.5]{Latz2020}, under the present assumptions on the solution operator \ref{ass:sol1}--\ref{ass:sol2} and the (strictly positive) noise distribution \ref{ass:noise1}--\ref{ass:noise3}, the solution $\mu^{y_{1:k}}$ of the BIP exists and is explicitly given by 
\begin{align}
\label{eq:bipsol}
\mu^{y_{1:k}}(d\bar{u}) 
=
\frac{1}{\cZ^\dagger_k(y_{1:k})}
\exp\left(
-\textstyle\sum_{j=1}^k \Phi^{\dagger,y_{j}}_j(\bar{u})
\right) \, \mup(d\bar{u}),
\end{align}
where 
\begin{align} \label{eq:Phidagger}
\Phi^{\dagger,y_{j}}_j(\bar{u})
:= 
-\log \rho(\cL^\dagger_{j}(\bar{u}) - y_{j}),
\end{align}
denotes the log-likelihood function, and
\begin{align} \label{eq:Zk}
\cZ^\dagger_k(y_{1:k})
=
\int_X 
\exp
\left(
-\textstyle\sum_{j=1}^k \Phi^{\dagger,y_j}_j(\bar{u})
\right)
\, \mup(d\bar{u}),
\end{align}
is the required normalization constant, which depends on $y_{1:k} = (y_1,\dots, y_k)$. We note that the condition that $\rho(y) > 0$ implies that the log-likelihood $\Phi^{\dagger,y_{1:k}}$ is finite, i.e., $\Phi^{\dagger,y_{1:k}}(\bar{u}) < \infty$ for all $\bar{u} \in X$. 

\begin{remark}[Gaussian noise]
If the noise $\eta \sim \Normal(0,\Gamma)$ is normally distributed (Gaussian), then (up to an unimportant additive constant)
\[
\Phi^{\dagger,y_j}(\bar{u}) = \frac12 |y_j-\L_j^\dagger(\bar{u})|^2_{\Gamma},
\]
where the natural $\Gamma$-norm is given by~\eqref{eq:Gnorm}.
 In this case, we have
\begin{align} 
\frac{d\mu^{y_{1:k}}}{d\mup}(\bar{u})
=
\frac{1}{\cZ_k^\dagger(y_{1:k})}
\exp\left(
-\frac 12
\sum_{j=1}^k
\left|
y_j - \L^{\dagger}_j(\bar{u})
\right|_{\Gamma}^2
\right).
\end{align}
\end{remark}

\subsubsection{Push-forward of BIP and stability}
In the next simple proposition, we note an explicit expression for the distribution \eqref{eq:DAsol} in terms of the solution of the corresponding BIP \eqref{eq:BIPsol}:
\begin{proposition}
\label{prop:da-bip}
Let $\cS^\dagger_t: X \to X$ be a Borel measurable mapping. Let $\nu^{y_{1:k}}_t$, $\mu^{y_{1:k}}$ be given by \eqref{eq:DAsol}, \eqref{eq:BIPsol}, respectively, for observables $\cG_j$ and with measurement noise satisfying \ref{ass:noise1}--\ref{ass:noise3}. Then 
\begin{align} \label{eq:da-bip}
\nu^{y_{1:k}}_{t} = \cS^\dagger_{t,\#} \mu^{y_{1:k}}.
\end{align}
\end{proposition}

\begin{proof}
Let $A \subset X$ be a Borel measurable set. We have 
\begin{align*}
\nu^{y_{1:k}}_t(A)
&=
\Prob\left[
\cS^\dagger_t(\bar{u}) \in A
\,|\,
y_1,\dots, y_k
\right]
\\
&=
\Prob\left[
\bar{u} \in [\cS^\dagger_t]^{-1}(A)
\,|\,
y_1,\dots, y_k
\right]
\\
&=
\mu^{y_{1:k}}([\cS^\dagger_t]^{-1}(A))
\\
&=
\cS^\dagger_{t,\#}\mu^{y_{1:k}}(A),
\end{align*}
As $A$ was arbitrary, it follows that $\nu^{y_{1:k}}_t = \cS^\dagger_{t,\#}\mu^{y_{1:k}}$. 
\end{proof}

Proposition \ref{prop:da-bip} specifies the relation between the Bayesian-DA problem and the corresponding BIP, via the push-forward under the solution operator. 

\begin{remark}
Based on Proposition \ref{prop:da-bip}, we also remark that the posteriors $\nu^{y_{1:k}}_t$ are indeed elements of $L^1_t(\cP)$, under the assumptions of the present work: If $\cS^\dagger_t$ satisfies \ref{ass:sol1}--\ref{ass:sol2}, then the push-forward $\mu \mapsto \cS^\dagger_{t,\#}\mu$ defines a well-defined map $\cP_1(X) \to L^1_t(\cP_1)$, and hence $\P[\cS^\dagger_t \in \slot \,|\, y_1,\dots, y_k] = \cS^\dagger_{t,\#}\mu^{y_{1:k}} \in L^1_t(\cP)$, if $\mup\in \cP_1(X)$.
\end{remark}

Furthermore, we note that Proposition \ref{prop:da-bip} immediately yields the following representation of the filtering distribution \eqref{eq:filt} in terms of the push-forward $\cS^\dagger_{t,\#}$, and the solutions of the BIP for the initial data $\mu^{y_{1:k}}$ in \eqref{eq:bipsol}:
\begin{align} \label{eq:dasol2}
\nu^{\bm{y}}_t = 
\begin{cases}
\cS^\dagger_{t,\#} \mu^{y_{1:(j-1)}}, & t \in [t_{j-1},t_j), \, j=1,\dots, N \\
\cS^\dagger_{t,\#} \mu^{y_{1:N}}, & t \ge t_N.
\end{cases}
\end{align}

At this point, we recall that the well-posedness of the BIP has been studied under very mild conditions on the forward operators and in a variety of metrics between probability measures in \cite{Latz2020,Sprungk2020}, including the Wasserstein distance. It is therefore natural to ask whether the results of \cite{Latz2020,Sprungk2020} can be used to obtain corresponding results also for the Bayesian data assimilation problem, based on their relationship \eqref{eq:da-bip} via the push-forward under $\cS^\dagger_t$? The following proposition indicates that, in general, bounds on the Wasserstein distance for the BIP (as obtained in e.g. \cite{Sprungk2020}) do \emph{not} automatically translate to corresponding bounds for Bayesian DA, if the solution operator is not sufficiently regular. For the straight-forward proof, we refer to Appendix \ref{sec:Wasserstein}, page \pageref{pf:LipW1}.

\begin{proposition}
\label{prop:LipW1}
Let $L_\cS := \Lip(\cS_t^\dagger: X\to X)$ denote the Lipschitz constant of the forward operator. Then 
\begin{align}
\label{eq:lip}
W_1\left(\cS^\dagger_{t,\#}\mu,\cS^\dagger_{t,\#}\mu'\right) 
\le 
L_\cS W_1(\mu, \mu'), 
\quad
\forall \, \mu,\mu'\in \cP_1(X),
\end{align}
and \emph{$L_\cS$ is optimal}: If $L>0$ is any other constant such that $W_1\left(\cS^\dagger_{t,\#}\mu,\cS^\dagger_{t,\#}\mu'\right) \le L W_1(\mu, \mu')$ for all $\mu,\mu'\in \cP_1(X)$, then $L \ge L_\cS$.
\end{proposition}

As shown in Proposition \ref{prop:LipW1}, the Wasserstein distance between two probability measures is amplified by the Lipschitz constant of the forward solution operator (for general measures $\mu,\mu'$). For the ill-posed problems considered in the present work, this Lipschitz constant is either very large or even infinite, $L_\cS=\infty$, and hence $W_1$-estimates  on the BIP -- even under the minimal assumptions of \cite{Sprungk2020, Latz2020} -- do not suffice to conclude similarly robust well-posedness results for Bayesian filtering problem. In view of applications to such ill-posed problems, it would be highly desirable to obtain estimates which are \emph{independent of the stability of the forward problem}. This is one of the main goals of the present work.

\subsubsection{Alternative representation of the filtering distribution}

We finally point out a different, recursive formulation of the Bayesian data assimilation problem, which is closer in spirit to filtering schemes such as the ensemble Kalman filter \cite{evensen2009data} or 3DVAR \cite{courtier1998ecmwf}.

\begin{remark}[Recursive filtering] 
\label{rmk:filtering}
Fix a prior measure $\mup$ at the initial time, and define 
\begin{align}\label{eq:filter-init}
\tilde{\nu}^{\dagger,y_{0}}_{t_0} := \cS^{\dagger}_{0,\#}\mup.
\end{align}
Given times $0=t_0 < t_1< \dots < t_N=T$ and measurements $y_1, \dots, y_N$, carry out the following two recursive steps. 
\begin{enumerate}
\item \define{Correction step:} Given $\tilde{\nu}^{y_{1:(j-1)}}_{t_{j-1}}$ as a prior at time $t_{j-1}$, solve the Bayesian inverse problem with new measurement 
$
y_{j} 
=
\G_{j}(\cS^\dagger_{t-t_{j-1}}(u)) + \eta_{j},
$
for $t \in [t_{j-1},t_{j}]$, to obtain a corrected Bayesian estimate 
\begin{align} \label{eq:correction}
\tilde{\nu}^{y_{1:j}}_{t_{j-1}}(du)
=
\frac{1}{Z^{\dagger}_{j}(y_{j})}
\exp
\left(
-\tilde{\Phi}^{\dagger,y_j}_j(u)
\right)
\, 
\tilde{\nu}^{y_{1:(j-1)}}_{t_{j-1}}(du),
\end{align}
where $\tilde{\Phi}^{\dagger,y_j}_j(u) := - \log\left(\cG_j(\cS^\dagger_{t-t_{j-1}}(u)) - y_j \right)$.

\item \define{Prediction step:} Based on this corrected estimate, predict the probability distribution at time $t_{j}$, as the push-forward:
\begin{align} \label{eq:prediction}
\tilde{\nu}^{y_{1:j}}_{t_{j}}
=
\cS^\dagger_{\delta t_{j},\#} 
\tilde{\nu}^{y_{1:j}}_{t_{j-1}},
\end{align}
where $\delta t_{j} = t_{j}-t_{j-1}$. 
\end{enumerate}

We note that, 
\begin{align} \label{eq:Phidagger2}
\Phi^{\dagger,y_j}_j(\bar{u}) = \tilde{\Phi}^{\dagger,y_j}_j(\cS^\dagger_{t_{j-1}}(\bar{u})),
\end{align}
by definition \eqref{eq:Phidagger} of $\Phi^{\dagger,y_j}_j$.
\end{remark}

We next observe that this recursive formulation is indeed equivalent to \eqref{eq:DAsol} (see Section \ref{sec:recursive}, p. \pageref{sec:recursive} for a proof):
\begin{proposition}
\label{prop:filterpf}
Assume that the solution operator satisfies $\cS^\dagger_s \circ \cS^\dagger_t = \cS^\dagger_{s+t}$ for all $s,t\ge 0$, in addition to \ref{ass:sol1}--\ref{ass:sol2}. 
Then the sequence $\tilde{\nu}^{y_{1:j}}_{t_j}$ obtained by the recursive prediction-correction procedure of Remark \ref{rmk:filtering} agrees with the filtering distribution \eqref{eq:filtsol}, i.e. we have $\tilde{\nu}^{y_{1:j}}_{t_j} = \nu^{\bm{y}}_{t_j}$, for all $j=1,\dots, N$.
\end{proposition}

In the present section, we have discussed the precise relation between the BIP and the Bayesian DA problem, showing that the data assimilation problem is a combination of a suitably formulated BIP for the initial data, followed by a prediction step. We finally point out that BIPs can be thought of as a special instance of the Bayesian DA (with trivial forward solution operator). This will allow us to translate certain results on data assimilation to the context of Bayesian inverse problems.

\begin{remark}[BIP as a special case of Bayesian DA]
\label{rmk:bip-as-da}
Set $\cS^\dagger_t(\bar{u}) := \bar{u}$ for all $t \in [0,T]$, and assume that all measurements are obtained at $t=0$. Then 
\[
\mu^{y_{1:N}}(du) = \nu^{y_{1:N}}_t(du) = \nu^{\bm{y}}_t(du),
\]
for all $t\in [0,T]$. Furthermore, we have for the discretized posterior
\begin{align*}
W_1(\mu^{y_{1:N}}, \mu^{\Delta,y_{1:N}})
&= 
\frac{1}{T} \, d_T\left(\nu^{\bm{y}}_t, \nu^{\Delta,\bm{y}}_t\right),
\\
W_1(\mu^{y_{1:N}}, \mu^{y'_{1:N}})
&= 
\frac{1}{T} \, d_T\left(\nu^{\bm{y}}_t, \nu^{\bm{y}'}_t\right).
\end{align*}
Hence, all results regarding the well-posedness, stability and consistency obtained for the Bayesian filtering setting in the present work, should readily imply corresponding results for the BIP setting, under the Wasserstein $W_1$-distance.
\end{remark}

\subsection{Well-posedness results for Bayesian DA}
\label{sec:wpBDA}

\subsubsection{General well-posedness result}

We can now state the following general well-posedness result for the Bayesian filtering problem, which shows that the filtering problem is well-posed under very mild boundedness assumptions, even if the corresponding forward problem is \emph{ill-posed}. Before stating our result, we recall that (cp. notation defined in appendix \ref{sec:Lp}), $\Vert \bar{u} \Vert_{L^1(\mup)} := \int_{X} \Vert \bar{u} \Vert_{X} \, d\mup(\bar{u})$. We then have:
\begin{theorem}[Filtering well-posedness] \label{thm:filtering-wp}
Let $\cS^\dagger_t: [0,T]\times X \to X$, $(t,\bar{u})\mapsto \cS^\dagger_t(\bar{u})$ be a Borel measurable solution operator, such that $\Vert \cS_t^\dagger (\bar{u}) \Vert_X \le B_\cS\Vert \bar{u} \Vert_X$ for all $t\in[0,T]$. Let $\mup\in \P_1(X)$ be a prior with finite first moment. Then the Bayesian filtering problem is well-posed: More precisely, the conditional probability $\nu^{\bm{y}}_t$ in \eqref{eq:filt} exists for any measurements $\bm{y} = (y_1,\dots, y_N) \in \R^{d \times N}$, $\nu^{\bm{y}}_t$ belongs to $L^1_t(\cP)$, and furthermore $\bm{y} \mapsto \nu^{\bm{y}}_t$ is stable, in the sense that for any $R>0$, there exists $C = C(R,\rho,N,B_\cS,\Vert \bar{u} \Vert_{L^1(\mup)}, T)>0$, such that
\begin{align} \label{eq:filtering-sup-stability}
W_1\left(
\nu^{\bm{y}}_t,
\nu^{\bm{y}'}_t
\right)
\le
C  |\bm{y}-\bm{y}'|_{\Gamma}, \quad \forall \, t\in [0,T],
\end{align}
and
\begin{align} \label{eq:filtering-stability}
\int_0^T
W_1\left(
\nu^{\bm{y}}_t,
\nu^{\bm{y}'}_t
\right)
\, dt
\le 
C |\bm{y}-\bm{y}'|_{\Gamma},
\end{align}
for all $\bm{y},\bm{y}'$ such that $|\bm{y}|_{\Gamma},\,|\bm{y}'|_{\Gamma}\le R$.
\end{theorem}

The proof of Theorem \ref{thm:filtering-wp} is provided in Section \ref{pf:filtering-wp} on page \pageref{pf:filtering-wp}.

\begin{example} [Filtering well-posedness for 2D Navier-Stokes]
As an immediate consequence of Theorem \ref{thm:filtering-wp}, we conclude that if $\cS^\dagger_t: L^2_x \to L^2_x$ is the solution operator of the incompressible Navier-Stokes equations \eqref{eq:NS} in two-dimensions, then the corresponding filtering problem is well-posed for \emph{any} viscosity $\nu > 0$, and the mapping $\bm{y} \mapsto \nu^{\bm{y}}_t$ from measurements to the solution is locally Lipschitz stable, with a constant that is \emph{independent} of the viscosity $\nu$. In contrast, we emphasize that the Lipschitz constant for the corresponding forward problem depends \emph{exponentially} on $1/\nu$ (cp. Theorem \ref{thm:NS2d}).
\end{example}

\begin{example} [Filtering well-posedness for 3D Navier-Stokes]
Similarly, for the three-dimensional Navier-Stokes equations we obtain a short-time well-posedness result if the prior $\mup$ is supported on sufficiently smooth initial data: e.g. if $\mup(\set{\bar{u}\in H^s_x}{\Vert \bar{u}\Vert_{H^s_x} \le M}) = 1$ for some $s>3/2$, and if the time-interval is sufficiently short $T< T^\ast(s,M)$, then the corresponding filtering problem $t\mapsto \nu^{\bm{y}}_t$ is well-posed on $t\in [0,T]$. Here, $H^s_x = H^s(\T^3;\R^3)$ denotes the well-known Sobolev space consisting of vector fields with square-integrable derivatives of order $s$.
\end{example}

\begin{example} [Uniform well-posedness for numerical discretizations in 2D and 3D]
Finally, we note that Theorem \ref{thm:filtering-wp} (with $\cS^\dagger_t$ replaced by $\cS^\Delta_t$) also implies the well-posedness of the filtering problem for numerical approximations, such as the spectral method introduced in Section \ref{sec:sv}, for any fixed $\nu, \Delta > 0$ and in \emph{both two and three dimensions}. Furthermore, the stability constant $C>0$ in \eqref{eq:filtering-stability} can be chosen \emph{uniformly}, for all values of $\nu, \Delta > 0$.
\end{example}

\subsubsection{Consistency}
Next, we discuss the consistency of approximate filtering based on a discretized solution operator $\cS^\Delta_t$, and the limiting filtering problem with solution operator $\cS^\dagger_t$. More precisely, we show that if $\cS^\Delta_t(\bar{u})\to \cS^\dagger_t(\bar{u})$ converges in a suitable sense, then $\nu^{\Delta,\bm{y}}_t \to \nu^{\bm{y}}_t$ in $L^1_t(\P)$ also converges:
\begin{theorem}[Filtering consistency] \label{thm:filtering-consistency}
Let $\mup \in \P_1(X)$ be a prior with finite second moments, $\Vert \bar{u} \Vert_{L^2(\mup)} < \infty$. Assume that $\cS^\Delta_t, \cS^\dagger_t: X\to X$ satisfy \ref{ass:sv1}--\ref{ass:sv3} and \ref{ass:sol1}--\ref{ass:sol2}, respectively. Then there exists a constant $C = C(R,T,N,\rho,\Vert \bar{u}\Vert_{L^2(\mup)}, \cG)> 0$, independent of $\Delta$, such that
\begin{align}
\label{eq:consistency}
\int_0^T
W_1
\left(
\nu^{\Delta,\bm{y}}_t
,
\nu^{\bm{y}}_t
\right)
\, dt
\le 
C 
\int_0^T 
\Vert \cS^\Delta_t(\bar{u}) - \cS^\dagger_t(\bar{u}) \Vert_{L^2(\mup)}
\, dt.
\end{align}
In particular, if $\cS^\Delta_t(\bar{u}) \to \cS^\dagger_t(\bar{u})$ in $L^1([0,T];L^2(\mup))$ at a certain convergence rate, then $\nu^{\Delta,\bm{y}}_t \to \nu_t^{\bm{y}}$ in $L^1_t(\P)$ converges at the same rate.
\end{theorem}

The proof of Theorem \ref{thm:filtering-consistency} is provided in Section \ref{pf:filtering-consistency} below, on page \pageref{pf:filtering-consistency}.

\begin{example}
Based the analysis of \cite{bardos2015stability}, the solutions computed by the numerical scheme in Section \ref{sec:sv} are expected to converge  spectrally for the two-dimensional Navier-Stokes equations: if $\mup(\set{\bar{u}\in H^s}{\Vert \bar{u}\Vert_{H^s}\le M})=1$ for some $M>0$, then $\Vert \cS^\Delta_t(\bar{u}) - \cS^\dagger_t(\bar{u})\Vert_{L^2_x} \le C \Delta^{s}$. In particular, by Theorem \ref{thm:filtering-consistency}, this implies a similar convergence rate also for the Bayesian filtering problem, i.e.
\begin{align*}
\int_0^T W_1(\nu^{\Delta,\bm{y}}_t, \nu^{\bm{y}}_t) \, dt
\le 
C \Delta^s.
\end{align*}
\end{example}

\begin{remark}[Surrogate models]
\label{rmk:surrogate}
Recently, surrogate models based on novel neural network architectures have been proposed to speed up many-query problems such as Bayesian data assimilation (see e.g. \cite{fourierop2020} for first results in this direction). These neural network-based methods provide an approximation $\cS^\Delta_t \approx \cS^\dagger_t$ of the underlying solution operator based on the minimization of an empirical loss function, which is chosen as a Monte-Carlo approximation of 
\[
\widehat{\cL}_{\mathrm{loss}}
=
\int_0^T 
\Vert \cS^\Delta_t(\bar{u}) - \cS^\dagger_t(\bar{u}) \Vert_{L^2(\mup)}^2
\, dt.
\]
Theorem \ref{thm:filtering-consistency} provides a first step towards a more detailed estimate on the approximation error of the underlying filtering problem $\nu^{\Delta,\bm{y}}_t\approx \nu^{\bm{y}}_t$, in terms of the loss $\widehat{\cL}_{\mathrm{loss}}$. Indeed, the upper bound \eqref{eq:consistency} implies the estimate 
$d_T(\nu^{\Delta,\bm{y}}_t,\nu^{\bm{y}}_t) \le C \sqrt{T\widehat{\cL}_{\mathrm{loss}}}$, on the time-integrated Wasserstein-distance between $\nu^{\Delta,\bm{y}}_t$ and $\nu^{\bm{y}}_t$!
\end{remark}

\subsubsection{Compactness and uniform stability}

We finally turn our attention to the approximate filtering problem for the particular case of the Navier-Stokes equations in 3D, in the \emph{absence} of a priori well-posedness for the forward problem. In contrast, the approximate solutions obtained from numerical discretizations, such as the spectral scheme outlined in section \ref{sec:sv} are well-defined for any given grid size $\Delta > 0$; Hence we focus our attention on the behaviour of these numerical discretizations $\cS^\Delta_t: X \to X$, with $X = L^2_x$ the space of square-integrable vector fields $\bar{u}: \T^3\to \R^3$ on the three-dimensional, $2\pi$-periodic torus $\T^3$, satisfying $\div(\bar{u}) = 0$. As pointed out in section \ref{sec:NStheory}, in this case, the uniqueness and stability of the forward problem for the Navier-Stokes equations is not known for general input data $\bar{u}\in L^2_x$. Nevertheless, the corresponding numerical approximations computed by the scheme \eqref{eq:spectralvisc} are well-defined for any $\Delta > 0$. Such approximations allow us to compute approximate filtering distributions $\nu^{\Delta,\bm{y}}_t$ for a given discretization parameter $\Delta > 0$. Even though stability of the corresponding forward problem is not known, and we could have $\Lip(\cS^\Delta_t) \to \infty$ as $\Delta \to 0$, the results of the present work nevertheless allow us to prove \emph{uniform stability} and \emph{compactness} for the corresponding approximate filtering distributions. 

Before stating the next theorem, we recall that a probability measure $\mup \in \cP_1(L^2_x)$ is said to have bounded support, if there exists $M>0$, such that 
\begin{align} \label{eq:bdsupp}
\mup(\set{\bar{u} \in L^2_x}{\Vert \bar{u} \Vert_{L^2_x} \le M})=1.
\end{align}
We can now state the following compactness result:

\begin{theorem}[Filtering compactness for Navier-Stokes] \label{thm:filtering}
Assume that the prior ${\mup}\in \P_1(L^2_x)$ has bounded support \eqref{eq:bdsupp} for some $M>0$, where $L^2_x := L^2(\T^d;\R^d)$ denotes the space of square integrable, periodic vector fields. Assume that $\mup$ is concentrated on divergence-free vector fields. Let $0=t_0 < t_1 < \dots < t_N=T$ be a strictly increasing sequence for fixed $N\in \mathbb{N}$. Let $\bm{y}=(y_1,y_2, \dots, y_N) \in \R^{d\times N}$ be a sequence of measurements. Let $\cS^\Delta_t: L^2_x\to L^2_x$ for $\Delta>0$ be approximate solution operators satisfying \ref{ass:sv1}--\ref{ass:sv3} of Proposition \ref{prop:sv}, and let $\nu^{\Delta,\bm{y}}_t$ be the solution of the associated filtering problem. Then the sequence $\nu^{\Delta,\bm{y}}_t$ is pre-compact in $C_\loc(\R^{d\times N};L^1_t(\P))$, as $\Delta \to 0$. In fact, for any $R > 0$, there exists a constant $C = C(R,\rho,N,M)>0$, such that 
\begin{align} \label{eq:W1stab}
\sup_{|\bm{y}|_\Gamma,|\bm{y}'|_\Gamma \le R}
W_1(\nu^{\Delta,\bm{y}}_t,\nu^{\Delta,\bm{y}'}_t) 
\le 
C |\bm{y} - \bm{y}'|_\Gamma,
\quad
\forall \, t\in [0,T],
\end{align}
and there exists a subsequence $\Delta_k \to 0$, and $\nu^{\ast,\bm{y}}_t$ such that for any $R>0$, 
\[
\sup_{|\bm{y}|_\Gamma \le R}
\int_0^T
W_1\left(
\nu^{\Delta_k,\bm{y}}_t,
\nu^{\ast,\bm{y}}_t
\right)
\, dt
\to 0,
\]
converges locally uniformly in $\bm{y}$. Any such limit satisfies the stability estimate \eqref{eq:W1stab} in $\bm{y}$.
\end{theorem}

For the details of the proof of Theorem \ref{thm:filtering}, we refer to Section \ref{pf:filtering}, page \pageref{pf:filtering}.

\begin{remark}[Stability of expectations] \label{rmk:Estability}
The stability estimate \eqref{eq:W1stab} in Theorem \ref{thm:filtering} is a consequence of the continuity properties of the noise distribution $\rho$, and is \emph{independent of any continuity properties of the observable $\L^\Delta(u)$}. One implication of \eqref{eq:W1stab} is that for any Lipschitz continuous $\Phi \in \Lip(X)$ and $t\in [0,T]$, the mapping
\[
\bm{y} \mapsto \E_t^{\Delta,\bm{y}}\left[\Phi\right]
:=
\frac{1}{Z^\Delta(\bm{y})} 
\int_X 
\Phi(u) \, d\nu_t^{\Delta,\bm{y}}(u),
\]
is locally Lipschitz continuous under the assumptions of Theorem \ref{thm:filtering}, i.e. for any $R>0$, there exists $C(R,\rho,N,M)>0$ , such that
\begin{align} \label{eq:Lipschitz-E}
\left|
\E^{\Delta,\bm{y}}\left[\Phi\right]
-
\E^{\Delta,\bm{y}'}\left[\Phi\right]
\right|
\le 
C\Vert \Phi \Vert_{\Lip} |\bm{y}-\bm{y}'|_\Gamma.
\end{align}
\end{remark}

\begin{remark}[Real-analyticity of expectations]
Under the assumptions of Theorem \ref{thm:filtering} and assuming additionally that the noise $\mu_{\mathrm{noise}} = \rho(y) \, dy$ is \emph{Gaussian noise}, then the Lipschitz continuity \eqref{eq:Lipschitz-E} of Remark \ref{rmk:Estability} can be considerably strengthened to show that, for any $\phi \in L^\infty(\mup)$ and $t\in [0,T]$, the mapping
\[
\R^{d\times N} \to \R,
\quad 
\bm{y} \mapsto \E_t^{\Delta,\bm{y}}\left[\phi\right],
\]
is \emph{real analytic}: indeed, for the corresponding Bayesian inverse problem (estimating the initial data) it follows from \cite[Lemma 4.5]{Herrmann_2020} that the mapping
\begin{align} \label{eq:hermann}
\R^{d \times N} \to \R, 
\quad
\bm{y} = (y_1,\dots, y_N) \mapsto 
\int_{X} \psi(\bar{u}) \, d\mu^{\Delta,y_{1:k}}(\bar{u}),
\end{align}
is real-analytic for any $\psi\in L^1(\mup)$ and $k \in \{1,\dots, N\}$. By \eqref{eq:dasol2}, we have 
\[
\nu^{\Delta,\bm{y}}_t= 
\sum_{k=0}^N 1_{[t_{k},t_{k+1})}(t) \; \cS^{\Delta}_{t,\#} \mu^{\Delta,y_{1:k}},
\]
where we formally set $t_{N+1}=\infty$, 
and hence for fixed $t\in [0,T]$, there exists $k$ such that $\nu^{\Delta,\bm{y}}_t = \cS^\Delta_{t,\#}\mu^{\Delta,y_{1:k}}$. Thus, for any $\phi \in L^\infty(\mup)$, we see that
\begin{align*}
\E_t^{\Delta,\bm{y}}\left[\phi\right]
&=
\int_{X} \phi(u) \, \nu^{\Delta,\bm{y}}_t(du)
=
\int_{X} \phi(u) \, \left[\cS^\Delta_{t,\#}\mu^{\Delta,y_{1:k}}\right](d u)
\\
&=
\int_{X} \phi\left( \cS^\Delta_t(\bar{u})\right) 
\, \mu^{\Delta,y_{1:k}}(d\bar{u}),
\end{align*}
is of the form \eqref{eq:hermann} with $\psi(\bar{u}) := \phi\left(\cS^\Delta_t(\bar{u})\right) \in L^\infty(\mup) \subset L^1(\mup)$. Hence $\bm{y}\mapsto \E^{\Delta,\bm{y}}_t[\phi]$ is real-analytic by the results of \cite{Herrmann_2020}.
In particular, this conclusion is independent of any regularity properties of $\bar{u} \mapsto \cS^\Delta_t(\bar{u})$. 
\end{remark}

\begin{remark} \label{rem:candidate}
Theorem \ref{thm:filtering} shows that even though the forward problem for the three-dimensional Navier-Stokes equations may be \emph{ill-posed}, we can nevertheless assign a set of candidate solutions for the Bayesian DA problem to a family of approximate posteriors $\nu^{\Delta,\bm{y}}_t$ at resolution $\Delta>0$. This set of candidate solutions in the limit $\Delta \to 0$ is given by
\[
\mathcal{M} = 
\set{
\nu^{\ast,\bm{y}}_t \in L^1_t(\P)
}{
\exists \Delta_k\to 0, \text{ s.t. } \nu^{\ast,\bm{y}}_t = \lim_{k\to \infty} \nu^{\Delta_k,\bm{y}}_t
},
\]
or equivalently, we can write
\[
\mathcal{M} = \bigcap_{\overline{\Delta} > 0} \cl\left({\set{(t,\bm{y})\mapsto \nu^{\Delta,\bm{y}}_t}{\Delta \le \overline{\Delta}}}\right),
\]
where $\cl$ denotes the closure in ${C_\loc(\R^d;L^1_t(\P))}$. We note that the set $\mathcal{M}$ is non-empty: This follows from the fact that any finite intersections are clearly non-empty and that each of the sets is a compact subset of ${C_\loc(\R^d;L^1_t(\P))}$. It then follows from the finite intersection property of compact sets that also their intersection $\mathcal{M} \ne \emptyset$, i.e. there always exists at least one candidate solution.
\end{remark}

The last remark can be interpreted as an \emph{existence result} for solutions of the Bayesian DA problem. This is an analogue of corresponding existence results for the forward problem of the Navier-Stokes equations \cite{Leray1934}. However, in contrast to the existence result of the forward problem, which implies the existence of suitable limits $\cS^{\Delta_k}_t(\bar{u}) \to \cS^\dagger_t(\bar{u})$ for fixed $\bar{u}$ and which may exhibit \emph{no stability in $\bar{u}$}, limits obtained for the filtering problem do not only exist, but are also \emph{uniformly stable} with respect to $\bm{y}$, giving rise to limits $\nu^{\ast,\bm{y}}_t$ with continuous dependence on $\bm{y}$. This remarkable stability of the data assimilation problem is in stark contrast with the corresponding forward problem, even though both problems involve the prediction of a future state.

We also note that, following the connection between Bayesian inverse problems and Bayesian data assimilation pointed out in Remark \ref{rmk:bip-as-da}, we can readily obtain a corresponding existence result for Bayesian inverse problems, which we state in passing:

\begin{theorem}[Compactness and stability for BIP]
\label{thm:bipcompact}
Let ${\mup}\in \P_1(X)$. Let $\cL^\Delta: X \to \R^d$ be a sequence of approximate numerical functionals, and let $\mu^{\Delta,y}$ be the solution of the associated BIP. If there exists $M>0$, such that $\sup_{\Delta} \Vert \cL^\Delta(\bar{u}) \Vert_{L^2(\mup)} \le M$, then the sequence $\mu^{\Delta,y}$ is pre-compact in $C_\loc(\R^{d};\cP_1(X))$, as $\Delta \to 0$: In fact, for any $R>0$, there exists $C = C(R,M,\rho)>0$, such that 
\[
W_1(\mu^{\Delta,y}, \mu^{\Delta,y'}) \le C |y-y'|_\Gamma, \quad 
\forall |y|_\Gamma, |y'|_\Gamma \le R.
\]
Furthermore, there exists a subsequence $\Delta_k \to 0$, such that for any $R>0$, 
\[
\sup_{|y|_\Gamma \le R}
W_1\left(
\mu^{\Delta_k,y},
\mu^{\ast,y}
\right)
\to 0,
\]
converges locally uniformly in $y$. Any such limit $\mu^{\ast,y}$ is locally Lipschitz continuous with respect to $y$, and can be represented in the form $d\mu^{\ast,y} = \frac{1}{Z^\ast(y)} \, \exp(-\Phi^\ast(\bar{u})) \, d\mup$ for a suitable functional $\Phi^\ast: X \to \R$.
\end{theorem}

We emphasize that in this case, the mere uniform boundedness of the mappings $\cL^\Delta: X \to \R^d$ is sufficient to obtain uniform stability and compactness. 

\section{Derivation of the main results}
\label{sec:derivation}

In this section, we provide the detailed mathematical derivation of the main results stated in the previous section. 

\subsection{Recursive filtering}
\label{sec:recursive}

We begin by providing a proof of the equivalence between the recursive filtering scheme of Remark \ref{rmk:filtering} and \eqref{eq:DAsol}.

\begin{proof}[Proof of Proposition \ref{prop:filterpf}]
We proceed by induction on $j$. The case $j=0$ is trivial, since
\[
\tilde{\nu}^{\emptyset}_{t_0}
\;
=
\;
\cS^\dagger_{t_0,\#} \mup
\;
=
\;
\cS^\dagger_{t_0,\#} \mu^{\emptyset}.
\]
 For $j\ge 1$, we integrate against an arbitrary, integrable (cylindrical) test function $\Psi(u)$ to find, with $\delta t_j = t_j-t_{j-1}$:
\begin{align*}
\int_{L^2_x} \Psi(u)
\,
\tilde{\nu}^{y_{1:j}}_{t_j}(du)
=
\int_{L^2_x} \Psi(u)
\,
\left[\cS^\dagger_{\delta t_{j},\#}\tilde{\nu}^{y_{1:j}}_{t_{j-1}}\right](du)
=
\int_{L^2_x} \Psi\left(\cS^\dagger_{\delta t_{j}}(u)\right)
\, \tilde{\nu}^{y_{1:j}}_{t_{j-1}}(du).
\end{align*}
Substitution of the correction step \eqref{eq:correction}, yields
\begin{align*}
\int_{L^2_x} \Psi(u)
\,
\tilde{\nu}^{y_{1:j}}_{t_{j-1}}(du)
=
\int_{L^2_x} \Psi\left(\cS^\dagger_{\delta t_{j}}(u)\right)
q^\dagger_j(u) 
\, \tilde{\nu}^{y_{1:(j-1)}}_{t_{j-1}}(du),
\end{align*}
where
\[
q^\dagger_j(u)
=
\frac{1}{Z_j^\dagger(y_j)}
\exp\left(
-\tilde{\Phi}^{\dagger,y_j}_j(u)
\right).
\]
By the induction hypothesis, the measure $\tilde{\nu}^{y_{1:(j-1)}}_{t_{j-1}}$ can be written as a push-forward: 
\[
\tilde{\nu}^{y_{1:(j-1)}}_{t_{j-1}}
=
\cS^\dagger_{t_{j-1},\#}
\mu^{y_{1:(j-1)}}.
\]
Thus, substituting above, we find
\begin{align*}
\int_{L^2_x} \Psi(u)
\,
\tilde{\nu}^{y_{1:j}}_{t_j}(du)
&=
\int_{L^2_x} \Psi(\cS^\dagger_{\delta t_j}(u))
\,
q^\dagger_j(u) \, 
\left[
\cS^{\dagger}_{t_{j-1},\#} \mu^{y_{1:(j-1)}}(u)
\right](du)
\\
&= 
\int_{L^2_x} \Psi(\cS^\dagger_{t_j}(\bar{u}))
q^\dagger_j(\cS^\dagger_{t_{j-1}}(\bar{u}))
\, 
\mu^{y_{1:(j-1)}}(d\bar{u}),
\end{align*}
where we have used that $\cS^\dagger_{t_{j-1}} \circ \cS^\dagger_{\delta t_j} = \cS^\dagger_{t_j}$ to simplify the argument of $\Psi$ in the last step. We now note that, by our definition of $q_j^{\dagger}$ and $\mu^{y_{1:(j-1)}}$, we have
\begin{align*}
q^{\dagger}_j(\cS^{\dagger}_{t_{j-1}}(\bar{u})) \, \mu^{y_{1:(j-1)}}(d\bar{u})
&\propto
\exp\left(
-\tilde{\Phi}_j^{\dagger,y_j}\left(\cS^\dagger_{t_{j-1}}(\bar{u})\right)
\right)
\\
&\qquad 
\times 
\exp\left(
-
\sum_{k=1}^{j-1} 
\Phi_k^{\dagger,y_k}(\bar{u})
\right) \, \mup(d\bar{u})
\\
&=
\exp\left(
-
\sum_{k=1}^{j} 
\Phi_k^{\dagger,y_k}(\bar{u})
\right) \, \mup(d\bar{u}),
\end{align*}
where we have taken into account the identity $\tilde{\Phi}^{\dagger,y_j}_j\left(\cS^\dagger_{t_{j-1}}(\bar{u})\right) = \Phi^{\dagger,y_j}_j(\bar{u})$ in the last step (cp. equation~\eqref{eq:Phidagger2}). The proportionality constant can be determined by normalization. The last expression is equal to $\mu^{y_{1:j}}$, and hence
\[
\int_{L^2_x} \Psi(u) \, \tilde{\nu}^{y_{1:j}}_{t_j}(du)
=
\int_{L^2_x} \Psi(\cS^\dagger_{t_j}(u)) \, \mu^{y_{1:j}}(du)
=
\int_{L^2_x} \Psi(u) \, \left[\cS^\dagger_{t_j,\#}\mu^{y_{1:j}}\right](du).
\]
Since $\Psi$ was an arbitrary (cylindrical) test function, the claimed identity follows.
\end{proof}

\subsection{Stability results for the BIP}
\label{sec:BIP}

The goal of the present section is to derive general stability results for the Bayesian inverse problem (BIP). Combined with the push-forward representation of the Bayesian filtering and smoothing distributions of Proposition \ref{prop:da-bip}, these estimates form the basis of our analysis of the well-posedness of the Bayesian DA problem. 

While the temporal nature of the measurement data is important for the filtering distribution $\nu^{\bm{y}}_t$, the solutions to the BIP of interest take the form $\mu^{y_{1:k}}$ of \eqref{eq:bipsol} for some given $k$, corresponding to a combined observable $\tL^\dagger: X \to \R^{\tilde{d}}$, $\tilde{d} = d \times k$, of the form $\tL^\dagger(\bar{u}) = (\cL^\dagger_1(\bar{u}),\dots,\cL^\dagger_k(\bar{u}))$, and with combined measurement $\tilde{y} = (y_1,\dots, y_k)$ of the form $\tilde{y}=\tL^\dagger(\bar{u}) + \tilde{\eta}$ with $\tilde{\eta} = (\eta_1,\dots, \eta_k)$. Hence, resulting in a standard BIP. Dropping the tildes in the following, we will thus study the general properties of posteriors for the BIP for measurable observables
\begin{align}
\label{eq:bipL}
\L^\dagger: \, X \to \R^d,
\qquad
\bar{u} \mapsto \L^\dagger(\bar{u}),
\end{align}
and with $\R^d$-valued noise $\eta \sim \rho(y) \, dy$ satisfying the assumptions \ref{ass:noise1}--\ref{ass:noise3}.
The solution of the BIP for measurement operator \eqref{eq:bipL} is then given \cite{Latz2020} by the posterior
\begin{align} \label{eq:posterior}
\mu^{y}(d\bar{u})
=
\frac{1}{Z^\dagger(y)}
\exp\left(
-\Phi^{\dagger,y}(\bar{u})
\right)
\, \mup(d\bar{u}),
\end{align}
where
\begin{align}\label{eq:loglike}
\Phi^{\dagger,y}(\bar{u})
:=
-\log \rho\left(
y - \L^\dagger(\bar{u})
\right)
\end{align}
denotes the log-likelihood function, and
\begin{align} \label{eq:Z}
Z^\dagger(y)
=
\int_X 
\exp
\left(
-\Phi^{\dagger,y}(\bar{u})
\right)
\, \mup(d\bar{u}),
\end{align}
is the required normalization constant. As is customary, we will denote the Radon-Nikodym derivative of $\mu^{y}$ with respect to $\mu$ by $d\mu^{y}/d\mup$, i.e.
\begin{align} \label{eq:dens}
\frac{d\mu^{y}}{d\mup}(\bar{u})
=
\frac{1}{Z^\dagger(y)}
\exp\left(
-\Phi^{\dagger,y}(\bar{u})
\right).
\end{align}
Given an approximation $\cL^\Delta \approx \cL^\dagger$, we similarly define $\mu^{\Delta,y}$, $\Phi^{\Delta,y}(\bar{u})$ and $Z^\Delta(y)$, with $\cL^{\Delta}(\bar{u})$ replacing $\cL^\dagger(\bar{u})$, in equations \eqref{eq:posterior}--\eqref{eq:Z}.

While the existence of a solution to the BIP is ensured by the non-negativity of the noise distribution $\rho(y)$, the stability and compactness results of the present work will be based on assumptions \ref{ass:noise1}--\ref{ass:noise3} on the noise. We begin our discussion by noting the following immediate observations from these assumptions:

\begin{lemma} \label{lem:noise}
Let $\cL^\dagger: X \to \R^d$ be any map. If the noise $\eta \sim \rho(y) \, dy$ satisfies assumptions \ref{ass:noise1}--\ref{ass:noise3}, then we have for all $y,y' \in \R^d$
\begin{align} \label{eq:phi1}
\left|
e^{-\Phi^{\dagger,y}(\bar{u})} - e^{-\Phi^{\dagger,y'}(\bar{u})}
\right|
\le 
L_\rho |y-y'|_\Gamma.
\end{align}
The log-likelihood $\Phi^{\Delta,y}$ is bounded from below, uniformly in $\Delta>0$ and $y\in \R^d$: there exists a constant $C_\rho \ge 0$ depending only on $\sup_{y\in \R^d}\rho(y) < \infty$, such that
\begin{align} \label{eq:phi3}
\essinf_{\bar{u}\in X} \Phi^{\dagger,y}(\bar{u}) \ge -C_\rho, 
\quad \forall \, y\in \R^d.
\end{align}
There exists a constant $B_\rho \ge 0$, such that
\begin{align} \label{eq:phi4}
\Phi^{\dagger,y}(\bar{u}) \le B_\rho + \frac12 |y-\L^\dagger(\bar{u})|_\Gamma^2.
\end{align}
In particular, we have
\begin{align} \label{eq:phi5}
\Phi^{\dagger,y}(\bar{u}) \le B_\rho + |y|_\Gamma^2 + |\L^\dagger(\bar{u})|_\Gamma^2.
\end{align}
Furthermore, if $\cL^\Delta$ is an approximation of $\cL^\dagger$, then for the same constant $L_\rho$ as above:
\begin{align} \label{eq:phi2}
\left|
e^{-\Phi^{\dagger,y}(\bar{u})} - e^{-\Phi^{\Delta,y}(\bar{u})}
\right|
\le 
L_\rho |\L^\dagger(\bar{u}) - \L^{\Delta}(\bar{u})|_\Gamma.
\end{align}
\end{lemma}

We now discuss the stability of the posterior $\mu^{y}$ for the BIP with respect to the measurement $y$. We note that our discussion of stability for the BIP overlaps in part with a similar discussion contained in \cite{Latz2020,Sprungk2020}. In particular, \cite{Sprungk2020} contains a general discussion of the stability of posteriors with respect to both the log-likelihood and priors, and with respect to a number of distance metrics between probability measures. Since some necessary estimates have not appeared in \cite{Latz2020,Sprungk2020}, at least in the precise form needed for our purposes, we have decided to include detailed proofs in this manuscript.

We begin our discussion of the stability properties of the BIP with the following lemma, proving that the sequence of densities $d\mu^{\Delta,y}/d\mup$ is uniformly bounded in $L^\infty({\mup})$. 

Recalling that $\Vert \cL^\dagger \Vert_{L^2(\mup)} := \left(\int_X \Vert \cL^\dagger(\bar{u}) \Vert_{X}^2 \, \mup(d\bar{u}) \right)^{1/2}$, we now state the following

\begin{lemma} \label{lem:Zbound}
Let $d\mu^{y}/d\mup$ be given by \eqref{eq:dens}, and $Z^\dagger(y)$ be defined as in \eqref{eq:Z}. Then 
\begin{align} \label{eq:Z-estimate}
Z^\dagger(y)
\ge \exp\left(-\int_X \Phi^{\dagger,y}(\bar{u}) \, \mup(d\bar{u}) \right),
\end{align}
and
\begin{align} \label{eq:dmu-estimate}
\frac{d\mu^{y}}{d\mup}(\bar{u})
\le 
\exp\left(\int_X \Phi^{\dagger,y}(\bar{u}) \, \mup(d\bar{u})
-
\essinf_{u\in X} \Phi^{\dagger,y}(\bar{u}) 
\right), \quad \forall \, \bar{u}\in X,
\end{align}
In particular, if the noise $\eta\sim \rho(y) dy$ satisfies \ref{ass:noise1}--\ref{ass:noise3}, then there exists a constant $C>0$ depending only on the noise distribution $\rho(y)$, such that 
\begin{align} \label{eq:Z-estimate2}
Z^\dagger(y) 
\ge \frac1C
\exp\left(-|y|_\Gamma^2-\Vert \L^\dagger \Vert_{L^2({\mup})}^2\right),
\end{align}
and 
\begin{align} \label{eq:dmu-estimate2}
\frac{d\mu^{\dagger,y}}{d\mup}(u)
\le 
C\exp\left(|y|^2_\Gamma +\Vert \L^\dagger \Vert_{L^2({\mup})}^2 \right),
\quad
\forall \,u \in X.
\end{align}
\end{lemma}

\begin{proof}
Since the exponential (Gaussian-like) factor in the definition of $d\mu^{\dagger,y}/d\mup$, eq. \eqref{eq:dens}, is bounded from above by $\exp(-\essinf_{u\in X} \Phi^{\dagger,y}(u))$, it suffices to prove the lower bound on $Z^\dagger(y)$. From the convexity of $z \mapsto \exp(-z)$ and Jensen's inequality, we obtain
\[
\exp\left(
- \int_{X} \Phi^{\dagger,y}(\bar{u}) \, \mup(d\bar{u})
\right)
\le
\int_{X} \exp(-\Phi^{\dagger,y}(\bar{u})) \, \mup(d\bar{u}) 
= Z^\dagger(y).
\]
This implies the first two estimates \eqref{eq:Z-estimate} and \eqref{eq:dmu-estimate} of this lemma. 

Under the noise assumptions \ref{ass:noise1}--\ref{ass:noise3}, and by \eqref{eq:phi5}, there exists $B_\rho>0$ depending only on the noise distribution $\rho(y)$, such the last term can be bounded from below, yielding
\begin{align*}
Z^\dagger(y)
\ge
\exp\left(-B_\rho -|y|_\Gamma^2 - \int_X |\L^\dagger(\bar{u})|_\Gamma^2 \, \mup(d\bar{u})\right),
\end{align*}
and thus the claimed inequality \eqref{eq:Z-estimate2} for $Z^\dagger(y)$ with $C = \exp(B_\rho)$.
%\sam{
Furthermore, by \eqref{eq:phi3}, there exists $C_\rho$, such that 
\[
\essinf_{\bar{u}\in X} \Phi^{\dagger,y}(\bar{u}) \ge -C_\rho.
\]
Thus the claimed inequality \eqref{eq:dmu-estimate2} holds with $C = \exp(B_\rho + C_\rho)$.
%}
\end{proof}

We next discuss the stability of $d\mu^{\dagger,y}/d\mup$ with respect to $y$. The following lemma shows that the map $y\mapsto d\mu^{\dagger,y}/d\mup$ is locally Lipschitz continuous with respect to the $L^\infty$-norm.

\begin{lemma}\label{lem:LinfLip}
Assume that the noise $\eta \sim \rho(y) \, dy$ satisfies \ref{ass:noise1}-\ref{ass:noise3}. Let $\L^\dagger(\bar{u}) \in L^2({\mup})$. There exists a constant $C>0$, depending only on the noise distribution $\rho(y)$, such that
\begin{align} \label{eq:LinfLip}
\left\Vert
\frac{d\mu^{\dagger,y}}{d{\mup}} - \frac{d\mu^{\dagger,y'}}{d{\mup}}
\right\Vert_{L^\infty({\mup})}
\le
C |y-y'|_\Gamma\exp\left({|y|_\Gamma^2+|y'|_\Gamma^2+2\Vert \L^\dagger\Vert_{L^2({\mup})}^2}\right).
\end{align}
\end{lemma}

\begin{proof}
Fix $\bar{u}\in X$ for the moment. Denote $e(y) := e(y;\bar{u}) = \exp(-\Phi^{\dagger,y}(\bar{u}))$, so that
\begin{align*}
\frac{d\mu^{\dagger,y}}{d{\mup}} - \frac{d\mu^{\dagger,y'}}{d{\mup}}
&= 
\frac{e(y)}{Z^{\dagger}(y)} - \frac{e(y')}{Z^{\dagger}(y')}
\\
&= 
\frac{e(y)-e(y')}{Z^\dagger(y)} + \frac{e(y')}{Z^\dagger(y')} \frac{(Z^\dagger(y')-Z^\dagger(y))}{Z^\dagger(y)}.
\end{align*}
By \eqref{eq:phi1}, we can estimate $|e(y) - e(y')|\le C |y-y'|_\Gamma$. Next, we note that this bound for $e(y)$ also implies that
\[
|Z^\dagger(y)-Z^\dagger(y')|
\le
\int_X |e(y;\bar{u})-e(y';\bar{u})| \, {\mup}(d\bar{u})
\le C |y-y'|_\Gamma \underbrace{\int_X 1 \, {\mup}(d\bar{u})}_{=1}.
\]
Hence, 
\begin{align*}
\left|
\frac{d\mu^{y}}{d{\mup}} - \frac{d\mu^{y'}}{d{\mup}}
\right|
&\le
\frac{C|y-y'|_\Gamma}{Z^\dagger(y)} + \frac{e(y')}{Z^\dagger(y')} \frac{C|y-y'|_\Gamma}{Z^\dagger(y)}.
\end{align*}
We proceed to estimate the factors multiplying $|y-y'|$ in the last two terms: From Lemma \ref{lem:Zbound}, we can estimate 
\[
\frac{1}{Z^\dagger(y)}
\le
C e^{|y|_\Gamma^2+\Vert \L^\dagger\Vert_{L^2({\mup})}^2} 
\le 
C e^{|y|_\Gamma^2 + |y'|_\Gamma^2+2\Vert \L^\dagger\Vert_{L^2({\mup})}^2},
\]
and
\[
\frac{e(y')}{Z^\dagger(y')} \frac{1}{Z^\dagger(y)}
\le
C e^{|y|_\Gamma^2 + |y'|_\Gamma^2+2\Vert \L^\dagger\Vert_{L^2({\mup})}^2}.
\]
Combining these estimates, we conclude that
\begin{align*}
\left|
\frac{d\mu^{y}}{d{\mup}} - \frac{d\mu^{y'}}{d{\mup}}
\right|
&\le
2C|y-y'|_\Gamma\exp\left({|y|_\Gamma^2+|y'|_\Gamma^2+2\Vert \L^\dagger\Vert_{L^2({\mup})}^2}\right).
\end{align*}
Since $\bar{u}\in X$ was arbitrary, the claimed inequality follows by taking the supremum over $\bar{u}\in X$ on the left.
\end{proof}

Let us also remark in passing the following Lemma, whose proof is analogous to the proof of Lemma \ref{lem:LinfLip}.

\begin{lemma}\label{lem:LinfLip2}
Assume that the noise $\eta \sim \rho(y) \, dy$ satisfies \ref{ass:noise1}--\ref{ass:noise3}. Let $\L^\dagger(\bar{u}), \L^\Delta(\bar{u}) \in L^2({\mup})$, and $y\in \R^d$, and denote the resulting (approximate) solution of the BIP by $\mu^{y}$, $\mu^{\Delta,y}$, respectively. There exists a constant $C>0$, depending only on the noise distribution $\rho(y)$, such that for any $p\in [1,\infty]$, we have
\begin{align*}
\Bigg\Vert
\frac{d\mu^{\Delta,y}}{d{\mup}} 
- 
\frac{d\mu^{y}}{d{\mup}}
\Bigg\Vert_{L^p(\mup)}
&\le
C \left\Vert 
\L^\Delta(\bar{u})-\L^\dagger(\bar{u})
\right\Vert_{L^p(\mup)}
\\
&\qquad 
\times \exp\left({2|y|_\Gamma^2+\Vert \L^\Delta\Vert_{L^2({\mup})}^2 + \Vert \L^\dagger\Vert_{L^2({\mup})}^2}\right).
\end{align*}
\end{lemma}

\begin{proof}
The proof is an almost verbatim repetition of the proof of Lemma \ref{lem:LinfLip},  with the roles of $y,y'$ and $\L^\Delta(\bar{u}), \L^\dagger(\bar{u})$ interchanged.
\end{proof}

\subsection{Stability results for Bayesian DA}

In this section, we investigate the stability properties of the solution of the filtering problem with respect to the measurements $y_1, \dots, y_N$. Our analysis will be based on the push-forward representation~\eqref{eq:da-bip} of the previous section and the stability results for the BIP in Section~\ref{sec:BIP}. We recall that the space $L^1_t(\P) = L^1([0,T];\P(X))$ defined in Section~\ref{sec:timep}, consists of all weak-$\ast$ measurable mappings $[0,T]\to \P(X)$, $t\mapsto \nu_t$, such that $\int_0^T \Vert u \Vert_{L^2_x} \, d\nu_t(u) \, dt < \infty$, with metric $d_T(\nu_t,\nu'_t) 
:=
\int_0^T 
W_1(\nu_t,\nu_t')
\, dt$.

We can now state the following lemma
\begin{lemma} \label{lem:filterstab}
Let $T>0$. Let $\mup\in \P_1(X)$ be a prior with finite first moments. Assume that the solution operator $\cS^\dagger_t: X\to X$ satisfies \ref{ass:sol1}--\ref{ass:sol2}, and let $\nu^{y_{1:j}}_t$ be given by \eqref{eq:DAsol} for $t\in [0,T]$. Then for any $R>0$, there exists $C=C(R,\rho,N,B_\cS, \Vert \bar{u} \Vert_{L^1(\mup)})>0$, such that for all $t, \delta t \ge 0$, with $t+\delta t \in [0,T]$, we have
\begin{align} \label{eq:supWest}
W_1\left(
\nu^{y_{1:j}}_\tau,
\nu^{y_{1:j}'}_\tau
\right)
\le 
C
\left(
\sum_{k=1}^j
\left|y_k - y_k'\right|_{\Gamma}^2
\right)^{1/2}
\end{align}
and
\begin{align} \label{eq:Westint}
\int_{t}^{t+\delta t}
W_1\left(
\nu^{y_{1:j}}_\tau,
\nu^{y_{1:j}'}_\tau
\right)
\, d\tau
\le 
C \delta t
\left(
\sum_{k=1}^j
\left|y_k - y_k'\right|_{\Gamma}^2
\right)^{1/2}
,
\end{align}
for all $y_{1:j} = (y_1, \dots, y_j)$, $y_{1:j}'=(y_1',\dots, y_j')$ such that 
\[
\left(
\sum_{k=1}^j
\left|y_k\right|_{\Gamma_k}^2
\right)^{1/2}
\le R,
\quad
\left(
\sum_{k=1}^j
\left|y_k'\right|_{\Gamma_k}^2
\right)^{1/2}
\le R.
\]
\end{lemma}

\begin{proof}
To simplify the notation in the following, we set
\[
|y_{1:j}|_{\Gamma}
:=
\left(
\sum_{k=1}^j
|y_k|_{\Gamma}^2
\right)^{1/2}.
\]
By Proposition \ref{prop:da-bip}, \eqref{eq:da-bip}, we have $\nu^{y_{1:j}}_t = \cS^\dagger_{t,\#}\mu^{y_{1:j}}$, where $\mu^{y_{1:j}}$ solves a BIP and is given by \eqref{eq:bipsol}. In fact, $\mu^{y_{1:j}}$ is the solution of a standard BIP with noise $\tilde{\eta} = (\eta_1, \dots, \eta_j)$. The noise distribution of $\tilde{\eta}$ satisfies assumption \ref{ass:noise1}--\ref{ass:noise3} with the distribution $\rho: \R^d \to \R$ replaced by $\tilde{\rho}: \R^{d\times j} \to \R$, where $\tilde{\rho}(y_1,\dots, y_j) := \prod_{k=1}^j \rho(y_k)$, and with altered constants in \ref{ass:noise1}--\ref{ass:noise3} depending now also on $j\le N$ in addition to $\rho$. Thus, by Lemma \ref{lem:LinfLip}, there exists a constant $C=C(\rho,N,R)>0$, depending only on the noise distribution, the total number of measurements $N$ and on $R>0$, such that we obtain
\begin{align} \label{eq:uprbound}
\left\Vert
\frac{
d\mu^{y_{1:j}}
}{
d\mup
}
-
\frac{
d\mu^{y_{1:j}'}
}{
d\mup
}
\right\Vert_{L^\infty(\mup)}
\le
C |y_{1:j} - y_{1:j}'|_{\Gamma}.
\end{align}

Let $\Phi(u) \in \Lip(X)$ be a function with Lipschitz constant $\le 1$. Then there exists $g(u)$ such that
\[
\Phi(u) - \Phi(0) 
=
g(u) \Vert u \Vert_{X},
\qquad 
|g(u)|\le 1.
\]
Now note that
\begin{align*}
\int_{X} \Phi(u) 
\left[
\nu_t^{y_{1:j}}(du) - \nu_t^{y_{1:j}'}(du)
\right]
&=
\int_{X} [\Phi(u)-\Phi(0)]
\left[
\nu_t^{\Delta,y_{1:j}}(du) - \nu_t^{\Delta,y_{1:j}'}(du)
\right]
\\
&=
\int_{X} g(u) \Vert u \Vert_{X}
\, 
\cS^{\dagger}_{t,\#}\left[
\mu^{y_{1:j}} - \mu^{y_{1:j}'}
\right](du)
\\
&=
\int_{X} g(\cS^{\dagger}_{t}(\bar{u}))\Vert S^{\dagger}_t(\bar{u})\Vert_{X}
\left[
\frac{d\mu^{y_{1:j}}}{d{\mup}} - \frac{d\mu^{y_{1:j}'}}{d{\mup}}
\right]
\, {\mup}(d\bar{u})
\\
&\le
\int_{X} \left|g(\cS^{\dagger}_{t}(\bar{u}))\right|\Vert \cS^{\dagger}_t(\bar{u})\Vert_{X}
\left|
\frac{d\mu^{y_{1:j}}}{d\mup} - \frac{d\mu^{y_{1:j}'}}{d\mup}
\right|
\, {\mup}(d\bar{u})
\\
&\explain{\le}{
\substack{
|g(u)|\le 1, \\
\Vert \cS^{\dagger}_t(\bar{u})\Vert_{X} \le B_\cS\Vert \bar{u} \Vert_{X}
}
}
B_\cS\int_{X} \Vert\bar{u} \Vert_{X}
\left|
\frac{d\mu^{y_{1:j}}}{d{\mup}} - \frac{d\mu^{y_{1:j}'}}{d{\mup}}
\right|
\, {\mup}(d\bar{u})
\\
&\le
B_\cS\left(\int_{X} \Vert \bar{u} \Vert_{X} \, {\mup}(d\bar{u})\right)
\left \Vert
\frac{d\mu^{y_{1:j}}}{d{\mup}} - \frac{d\mu^{y_{1:j}'}}{d{\mup}}
\right \Vert_{L^\infty({\mup})}.
\end{align*}
Taking the supremum over all $\Phi(u)$ such that $\Vert \Phi \Vert_{\Lip} \le 1$ on the left, and noting the upper bound \eqref{eq:uprbound} on the last term, we find
\[
W_1
\left(
\nu^{y_{1:j}}_t, \nu^{y_{1:j}'}_t
\right)
\le 
C |y_{1:j} - y_{1:j}'|_{\Gamma},
\]
where the constant $C = C(R,\rho,N,B_\cS, \Vert \bar{u} \Vert_{L^1(\mup)}) > 0$ depends on $R$, the noise distribution $\rho$, the number of measurements $N$, the first moment of the prior $\mup$ and on the boundedness constant $B_\cS$ of the forward operator $\cS^\dagger_t$ (cp. \ref{ass:sol2}). In particular, $C$ is independent of $y_{1:j}$, $y_{1:j}'$. This shows the upper bound \eqref{eq:supWest}. Integrating in time, we obtain the claimed inequality \eqref{eq:Westint}
\[
\int_{t}^{t+\delta t}
W_1\left(
\nu^{y_{1:j}}_t, \nu^{y_{1:j}'}_t
\right)
\, dt
\le 
C\delta t |y_{1:j} - y_{1:j}'|_{\Gamma}.
\]
\end{proof}

While the above estimate is applicable to the smoothing problem, i.e. with a fixed set of measurements for all $t\in [0,T]$, we will next prove a corresponding stability theorem for the solution of the \emph{filtering} problem; more precisely, we prove that if $\nu^{\bm{y}}_t$ denotes the solution of the filtering problem with prior $\mup\in \P_1(X)$, for a solution operator $\cS^\dagger_t: X\to X$ satisfying \ref{ass:sol1}--\ref{ass:sol2}, and measurements $\bm{y} = (y_1,\dots, y_N)$, then for any $R>0$, there exists $C = C(R,\rho,N,B_\cS, \Vert \bar{u} \Vert_{L^1(\mup)}) > 0$, such that
\begin{align} 
W_1\left(
\nu^{\bm{y}}_t,
\nu^{\bm{y}'}_t
\right)
\le 
C |\bm{y}-\bm{y}'|_{\Gamma},
\quad
\forall \, t\in [0,T],
\end{align}
and
\begin{align} 
\int_0^T
W_1\left(
\nu^{\bm{y}}_t,
\nu^{\bm{y}'}_t
\right)
\, dt
\le 
CT |\bm{y}-\bm{y}'|_{\Gamma},
\end{align}
for all $\bm{y},\bm{y}'$ such that $|\bm{y}|_{\Gamma},\,|\bm{y}'|_{\Gamma}\le R$. 

\begin{proof}[Proof of Theorem \ref{thm:filtering-wp}]
\label{pf:filtering-wp}
The claimed stability estimate follows readily from Lemma~\ref{lem:filterstab}: Indeed, $\nu^{\bm{y}}_t$ is defined piece-wise in time, for $t\ge 0$, as
\[
\nu^{\bm{y}}_t
=
\sum_{k=0}^N  1_{[t_{k},t_{k+1})}(t)\,\nu^{y_{1:k}}_{t},
\]
where we formally define $t_{N+1}:=\infty$. By Lemma \ref{lem:filterstab}, equation \eqref{eq:supWest}, this implies that for any $|\bm{y}|_\Gamma,|\bm{y}'|_\Gamma \le R$, we have
\begin{align*}
\sup_{t\in [0,T]}
W_1\left(\nu^{\bm{y}}_t,\nu^{\bm{y}'}_t\right) 
&=
\max_{k=0,\dots, N} W_1\left(\nu^{y_{1:k}}_t,\nu^{y_{1:k}'}_t\right) 
\\
&\le
C \max_{k=0,\dots, N}
|y_{1:k} - y_{1:k}'|_{\Gamma}
\\
&\le
C |\bm{y} - \bm{y}'|_{\Gamma},
\end{align*}
where we recall that the constant $C>0$ in Lemma \ref{lem:filterstab} depends on $R$, $\rho$, $N$, $B_\cS$ and $\Vert \bar{u} \Vert_{L^1(\mup)}$. This last estimate immediately implies
\begin{align*}
\int_0^T
W_1\left(
\nu^{\bm{y}}_t,
\nu^{\bm{y}'}_t
\right)
\, dt
&\le
C T |\bm{y} - \bm{y}'|_{\Gamma},
\end{align*}

\end{proof}

\subsection{Consistency results for Bayesian DA}

In the present subsection, we discuss the consistency of the approximate filtering problems based on the discretized solution operator $\cS^\Delta_t$, and the limiting filtering problem with solution operator $\cS^\dagger_t$ (cp. Theorem \ref{thm:filtering-consistency}). More precisely, we show that  if $\mup\in \cP_1(X)$ has finite second moments, if the noise distribution $\rho(y) \, dy$ satisfies \ref{ass:noise1}--\ref{ass:noise3} and if the observables $\cG_j: L^1([0,T];X) \to \R^d$ are Lipschitz continuous, then for any (approximate) solution operators $\cS^\dagger_t, \cS^\Delta_t: X \to X$ satisfying a uniform estimate
\begin{align} \label{eq:boundedness1}
\Vert \cS^\dagger_t(\bar{u}) \Vert_{X}, \Vert \cS^\Delta_t(\bar{u}) \Vert_{X} \le B_\cS \Vert \bar{u}\Vert_{X}, 
\end{align}
we have 
\begin{align}
\int_0^T
W_1
\left(
\nu^{\Delta,\bm{y}}_t
,
\nu^{\bm{y}}_t
\right)
\, dt
\le 
C 
\int_0^T 
\Vert \cS^\Delta_t(\bar{u}) - \cS^\dagger_t(\bar{u}) \Vert_{L^2(\mup)}
\, dt.
\end{align}

Before coming to the proof of this claim, we note that, under assumption \ref{ass:sol2} and for Lipschitz continuous observables $\cG_j: L^1([0,T];X)\to \R^d$ (cp. equation \ref{eq:Lj}), we have for $\cL_j^\Delta(\bar{u}) = \cG_j(\cS^\Delta(\bar{u}))$, $\cL_j^\dagger(\bar{u}) = \cG_j(\cS^\dagger(\bar{u}))$:
\begin{align} \label{eq:Llip}
\begin{aligned}
\Vert \L^\Delta_j(\bar{u}) \Vert_{L^2(\mup)},
\Vert \L^\dagger_j(\bar{u}) \Vert_{L^2(\mup)}
&\le
C
\left(
1+\Vert \bar{u} \Vert_{L^2(\mup)}
\right),
\\
\Vert \L^\Delta_j(\bar{u}) - \L^{\dagger}_j(\bar{u}) \Vert_{L^2(\mup)}
&\le
C\int_{t_{j-1}}^{t_j} \Vert \cS^\Delta_t(\bar{u}) - \cS^{\dagger}_t(\bar{u}) \Vert_{L^2(\mup)} \, dt,
\end{aligned}
\end{align}
for all $j=1,\dots, N$, where $C = C(\G, B_\cS, T)>0$ depends only on $T$, the (Lipschitz continuous) observables $\cG = (\cG_1,\dots, \cG_N)$, and the boundedness constant $B_\cS$ in \eqref{eq:boundedness1}.

\begin{proof}[Proof of Theorem \ref{thm:filtering-consistency}]
\label{pf:filtering-consistency}
By Proposition \ref{prop:da-bip} and \eqref{eq:dasol2}, we have
\[
\left.
\begin{aligned}
\nu^{\bm{y}}_t
&=
\cS^\dagger_{t,\#}
\mu^{y_{1:j}}, 
\\
\nu^{\Delta,\bm{y}}_t
&=
\cS^\Delta_{t,\#}
\mu^{\Delta,y_{1:j}}, 
\end{aligned}
\right\}
\qquad 
\forall \, t \in [t_{j},t_{j+1}).
\] 
Given $\Phi \in \Lip$, with $\Vert \Phi \Vert_{\Lip} \le 1$ and $\Phi(0) = 0$, we find
\begin{align*}
\int_{X} \Phi(u) \left[
\nu^{\Delta,\bm{y}}_t(du)
-
\nu^{\bm{y}}_t(du)
\right]
&=
\int_{X} \Phi(u) \left[
\cS^{\Delta}_{t,\#}\mu^{\Delta,y_{1:j}}(du)
-
\cS^\dagger_{t,\#}\mu^{y_{1:j}}(du)
\right]
\\
&=
\int_{X} \Phi(u) 
\left[
\cS^{\Delta}_{t,\#}\mu^{\Delta,y_{1:j}}(du)
-
\cS^\Delta_{t,\#}\mu^{y_{1:j}}(du)
\right]
\\
&\qquad
+
\int_{X} \Phi(u) 
\left[
\cS^{\Delta}_{t,\#}\mu^{y_{1:j}}(du)
-
\cS^\dagger_{t,\#}\mu^{y_{1:j}}(du)
\right]
\\
&=: (I) + (II)
.
\end{align*}
We can estimate the two last terms individually as follows: For the first term, we obtain
\begin{align*}
(I)
&=
\int_{X} \Phi(u) 
\left[
\cS^{\Delta}_{t,\#}\mu^{\Delta,y_{1:j}}(du)
-
\cS^\Delta_{t,\#}\mu^{y_{1:j}}(du)
\right]
\\
&=
\int_{X} \Phi(\cS^\Delta_t(\bar{u}))\left[
\frac{d\mu^{\Delta,y_{1:j}}}{d{\mup}}
-
\frac{d\mu^{y_{1:j}}}{d{\mup}}
\right] 
\, {\mup}(d\bar{u})
\\
&\le 
C\int_{X} \Vert \bar{u} \Vert_{X}
\left|
\frac{d\mu^{\Delta,y_{1:j}}}{d{\mup}}
-
\frac{d\mu^{y_{1:j}}}{d{\mup}}
\right| \, {\mup}(d\bar{u})
\\
&\le 
C\Vert \bar{u} \Vert_{L^2(\mup)}
\left\Vert
\frac{d\mu^{\Delta,y_{1:j}}}{d{\mup}}
-
\frac{d\mu^{y_{1:j}}}{d{\mup}}
\right\Vert_{L^2(\mup)} .
\end{align*}
The last term can be estimated using Lemma~\ref{lem:LinfLip2}, recalling that $\mu^{\Delta,y_{1:j}}$ is defined as the posterior with prior $\mup$ and given the measurements $\cL^\Delta_{1:j} = (\L^\Delta_1,\dots, \L^\Delta_j)$, with $
y_{1:j} = \cL^\Delta_{1:j} + \eta_{1:j}$ and $\eta_{1:j} = (\eta_1,\dots, \eta_j) \sim \rho(y_1) \, dy_1 \otimes \dots \otimes \rho(y_j) \, dy_j$. Lemma \ref{lem:LinfLip2} therefore yields
\begin{align*}
\left\Vert
\frac{d\mu^{\Delta,y_{1:j}}}{d{\mup}}
-
\frac{d\mu^{y_{1:j}}}{d{\mup}}
\right\Vert_{L^2(\mup)}
&\le 
C
\Vert \L^\Delta_{1:j}(\bar{u}) - \L^\dagger_{1:j}(\bar{u}) \Vert_{L^2(\mup)}
\\
&=
C
\left(
\sum_{\ell=1}^j
\Vert \L^\Delta_\ell(\bar{u}) - \L^\dagger_\ell(\bar{u}) \Vert_{L^2(\mup)}^2
\right)^{1/2},
\end{align*}
for some constant $C = C(R,T,N,\rho,\Vert \bar{u}\Vert_{L^2(\mup)}, \cG, T)>0$ depending on $R$, $T$, the noise distribution $\rho(y)$, the prior $\mup$ and on the observables $\cG = (\cG_1,\dots, \cG_N)$; here, we have used the fact that $y_{1:j}$ is fixed, and that $\Vert \L_\ell^\Delta(\bar{u})\Vert_{L^2(\mup)}, \Vert \L^\dagger_\ell(\bar{u}) \Vert_{L^2(\mup)} \le C (1+\Vert \bar{u} \Vert_{L^2(\mup)}) < \infty$ are bounded independently of $\Delta > 0$ for a constant $C = C(\cG,T)>0$ (cp. equation \eqref{eq:Llip}). The latter observation allows us to bound the additional exponential factor in Lemma \ref{lem:LinfLip2} uniformly in $\Delta$.
Continuing, we note that the observables are Lipschitz continuous by assumption \eqref{eq:Llip}; we have
\[
\Vert 
\L^\Delta_\ell(\bar{u})
-
\L^\dagger_\ell(\bar{u})
\Vert_{L^2(\mup)}
\le 
C
\int_{t_{\ell-1}}^{t_{\ell}}
\Vert 
\cS^\Delta_{t}(\bar{u}) 
-
\cS^\dagger_t(\bar{u})
\Vert_{L^2(\mup)}
\, dt,
\]
where $C = C(\G) > 0$. It follows that
\begin{align*}
(I)
\le
C
\left(
\sum_{\ell=1}^j
\left[
\int_{t_{\ell-1}}^{t_{\ell}}
\left\Vert \cS^\Delta_t(\bar{u}) - \cS^\dagger_t(\bar{u}) \right\Vert_{L^2({\mup})}
\, 
dt
\right]^2
\right)^{1/2}
.
\end{align*}
Denoting $F(t,\ell) := 1_{[t_{\ell-1},t_{\ell})}(t) \Vert \cS^\Delta_t(\bar{u}) - \cS_t^\dagger(\bar{u}) \Vert_{L^2(\mup)}$, we can estimate the last term as follows, using Minkowski's integral inequality:
\begin{align*}
C\left(
\sum_{\ell=1}^j
\left[
\int_{0}^{T}
F(t,\ell)
\, 
dt
\right]^2
\right)^{1/2}
\le
C\int_{0}^{T}
\left(
\sum_{\ell=1}^j
|F(t,\ell)|^2
\right)^{1/2}
\, 
dt.
\end{align*}
Finally, recalling that all $F(t,\ell)$, $\ell=1,\dots, j$, have \emph{disjoint} supports in $t$, we conclude that
\begin{align*}
(I) 
&\le
C\int_{0}^{T}
\left(
\sum_{\ell=1}^j
|F(t,\ell)|^2
\right)^{1/2}
\, 
dt
=
C
\sum_{\ell=1}^j
\int_{t_{\ell-1}}^{t_{\ell}}
|F(t,\ell)|
\, 
dt
\\
&\le
C\int_0^T
\Vert \cS^\Delta_t(\bar{u}) - \cS^\dagger_t(\bar{u}) \Vert_{L^2(\mup)}
\, dt.
\end{align*}

To estimate the second term, we note that
\begin{align*}
\int_{X} \Phi(u) \left[
\cS^{\Delta}_{t,\#}\mu^{y_{1:j}}(du)
-
\cS_{t,\#}^\dagger\mu^{y_{1:j}}(du)
\right]
&=
\int_{X} \left[\Phi(\cS^\Delta_t(\bar{u}))-\Phi(\cS^\dagger_t(\bar{u}))\right] \, \mu^{y_{1:j}}(d\bar{u})
\\
&\le
\int_{X} \left\Vert \cS^\Delta_t(\bar{u})-\cS^\dagger_t(\bar{u})\right\Vert_{X} \, \mu^{y_{1:j}}(d\bar{u})
\\
&\le
C\int_{X} \left\Vert \cS^\Delta_t(\bar{u})-\cS^\dagger_t(\bar{u})\right\Vert_{X} \, {\mup}(d\bar{u})
\\
&\le
C\left\Vert \cS^\Delta_t(\bar{u}) - \cS_t^\dagger(\bar{u}) \right\Vert_{L^2({\mup})}.
\end{align*}
Thus, employing the above estimates for $(I)$ and $(II)$, we conclude that for any $\Phi \in \Lip$, $\Vert \Phi\Vert_{\Lip}\le 1$, and for any $t\in [0,T]$, we have
\begin{align*}
\int_{X} \Phi(u) \left[
\nu^{\Delta,\bm{y}}_t(du)
-
\nu^{\bm{y}}_t(du)
\right]
&\le
C
\left\Vert
\cS^\Delta_t(\bar{u}) - \cS^\dagger_t(\bar{u})
\right\Vert_{L^2(\mup)}
\\
&\qquad
+
C
\int_{0}^{T}
\left\Vert
\cS^\Delta_t(\bar{u}) - \cS^\dagger_t(\bar{u})
\right\Vert_{L^2(\mup)}
\, dt.
\end{align*}
Taking the supremum over all such $\Phi$ on the left, and integrating over $t\in [0,T]$, it follows that
\begin{align*}
\int_0^T
W_1\left(
\nu^{\Delta,\bm{y}}_t
,
\nu^{\bm{y}}_t
\right)
\, dt
&\le
C
\int_0^T
\left\Vert
\cS^\Delta_t(\bar{u}) - \cS^\dagger_t(\bar{u})
\right\Vert_{L^2(\mup)}
\, 
dt
,
\end{align*}
where $C = C(R,T,N,\rho,\Vert \bar{u}\Vert_{L^2(\mup)}, \cG)>0$ is independent of $\Delta$.
\end{proof}

\subsection{Compactness results for Bayesian DA}
\label{sec:bda-compact}

In the present section, we will prove a compactness result, which applies in particular to the numerical approximations of the Navier-Stokes equations introduced in Section \ref{sec:sv}. We recall that, by Proposition \ref{prop:sv}, numerical solutions computed e.g. by suitable spectral schemes satisfy the following properties:
\begin{itemize}
\item[\namedlabel{ass:sv1p}{($\Delta$.1')}] \emph{Energy bound:} There exists $C>0$, independent of $\Delta>0$, such that
%\begin{align} \label{eq:sv1}
$\Vert \cS^\Delta_t(\bar{u}) \Vert_{L^2_x} \le \bar{B}_\cS \Vert \bar{u} \Vert_{L^2_x},$
%\end{align}
\item[\namedlabel{ass:sv2p}{($\Delta$.2)}] \emph{Coercivity:}
%\begin{align} \label{eq:sv2}
$\int_0^T \Vert \nabla u^\Delta(t) \Vert^2_{L^2_x} \, dt 
\le \nu^{-1} \Vert \bar{u} \Vert_{L^2_x}^2,
$
%\end{align}
\item[\namedlabel{ass:sv3p}{($\Delta$.3)}] \emph{Weak time-regularity:}
There exists $C = C(\Vert \bar{u} \Vert), L > 0$, such that 
%\begin{align} \label{eq:sv3}
$\Vert u^\Delta(t) - u^\Delta(s) \Vert_{H^{-L}_x}
\le 
C |t-s|.
$
%\end{align}
In particular, we have $u^\Delta \in \Lip([0,T];H^{-L}_x)$ uniformly in $\Delta >0$.
\end{itemize}
When re-stating the first property, we have slightly relaxed \ref{ass:sv1}, allowing for a uniformly bounded constant $C>0$ in \ref{ass:sv1p}.

 Our compactness result is motivated by the study of statistical solutions of the compressible and incompressible Euler equations in \cite{FLMW1,LMP1,LMP2}. There, it is shown that under a suitable average regularity condition, the sequence of discretized approximate solutions $\mu^\Delta_t := \cS^\Delta_{t,\#} {\mu}$ (push-forward by the discretized solution operator) is compact in $\P_1(L^2_x)$, provided that the following measure of average two-point correlations
\begin{align} \label{eq:structfun}
\S_2^T(\mu^\Delta_t;r) 
:=
\left(
\int_0^T
\int_{L^2_x} 
\S_2(u;r)^2
\, \mu^\Delta_t(du)
\, dt
\right)^{1/2},
\end{align}
are uniformly bounded as $\Delta \to 0$, where 
\begin{align}
\S_2(u;r)
:=
\left(
\int_D \fint_{B_r(0)} 
|u(x+h)-u(x)|^2
\, dh \, dx
\right)^{1/2},
\end{align}
measures the average of two-point correlations of $u$. The quantity $r \mapsto \S_2^T(\mu^\Delta_t;r)$ is referred to as the (time-integrated) structure function of $\mu^\Delta_t$. For simplicity, we will state the following results in the periodic setting with domain $D = \T^n \simeq [0,2\pi]^n$ the $2\pi$-periodic torus, and with $L^2_x := L^2(\T^n;\R^n)$ the space of square-integrable vector fields on $\T^n$. More precisely, we state the following proposition, which follows from \cite[Theorem 2.2]{LMP1}:
\begin{proposition}
\label{prop:cpct}
Let $\mu^\Delta_t$ be of the form $\mu^\Delta_t = \cS^{\Delta}_{t,\#}\mup$, where $\cS^\Delta_t: L^2_x\to L^2_x$ satisfies assumptions \ref{ass:sv1p},\ref{ass:sv3p}, and where $\mup \in \P_1(L^2_x)$ has bounded support, i.e. there exists $M>0$, such that 
$
\mup\left(\set{\bar{u}\in L^2_x}{\Vert \bar{u} \Vert_{L^2_x} \le M}\right) = 1.
$
If there exists a modulus of continuity $\phi(r)$ 
%(i.e., $\phi(r) > 0, ~\forall r$ and $\lim\limits_{r \rightarrow 0} \phi(r) = 0$)
,  such that
\[
\sup_{\Delta} \S_2^T(\mu^\Delta_t;r) \le \phi(r),
\quad
\forall\, r \ge 0,
\]
then $\mu^{\Delta}_t$ is compact in $L^1_t(\P)$.
\end{proposition}
Numerical evidence for the uniform boundedness of these structure functions $\S^T_2(\mu^\Delta_t;r)$ has been presented for a variety of initial probability measures ${\mup} \in \cP_1(L^2_x)$ for the incompressible Euler equations (i.e. in the zero-viscosity limit of the Navier-Stokes equations) in \cite{LMP1,LMP2}, and in the context of hyperbolic conservation laws in \cite{FLMW1}. While these results of \cite{LMP1,LMP2,FLMW1} were mostly based on empirical observations, in the present case of the Navier-Stokes equations, we will \emph{rigorously prove} that the structure functions \eqref{eq:structfun} are uniformly bounded as $\Delta \to 0$ (cp. Lemma \ref{lem:coerc}, below). We will furthermore extend the compactness result summarized in Proposition \ref{prop:cpct} to the Bayesian filtering context, when  the underlying model is combined with additional measurements.

We formulate the numerically observed \cite{FLMW1,LMP1,LMP2} bounds on $\S^T_2(\mu^\Delta_t;r)\le \phi(r)$ abstractly as the following assumption:
\begin{assumption} \label{ass:pushforward}
The prior $\mup \in \cP_1(L^2_x)$ has bounded support, i.e. there exists $M>0$, such that 
\[
\mup\left(
\set{\bar{u}\in L^2_x}{\Vert \bar{u} \Vert_{L^2_x} \le M}
\right) = 1,
\]
and there exists a modulus of continuity $\phi(r)$, such that 
\begin{align} \label{eq:pushforward}
\S_2^T(\cS^\Delta_{t,\#}{\mup};r)
\le 
\phi(r), \quad \forall r > 0, \; t \in [0,T],
\end{align}
uniformly for all $\Delta > 0$. Here $\cS^\Delta_{t,\#}{\mup}$ denotes the push-forward measure of the prior $\mup$ by the discretized solution operator $\cS^\Delta_t$.
\end{assumption}

The next proposition shows that the structure function bound \eqref{eq:pushforward} is automatically satisfied for numerical schemes satisfying \ref{ass:sv2p}:
\begin{lemma}
\label{lem:coerc}
Let $\mup\in \P_1(L^2_x)$ be a prior for the incompressible Navier-Stokes equations \eqref{eq:NS}, such that $\int_{L^2_x} \Vert \bar{u} \Vert^2_{L^2_x} \, \mup(d\bar{u})< \infty$. Let $\cS^\Delta_t: L^2_x \to L^2_x$ denote the approximate solution operator obtained from the spectral scheme \eqref{eq:spectralvisc}. Then we have the following a priori structure function bound:
\[
\S_2^T(\cS^\Delta_{t,\#} \mup; r) 
\le 
\frac{r}{\sqrt{2\nu}} \left(\int_{L^2_x} \Vert \bar{u} \Vert^2_{L^2_x} \, \mup(d\bar{u})\right)^{1/2}.
\]
In particular, $\S_2^T(\cS^\Delta_{t,\#} \mup; r) \le Cr$ is uniformly bounded by a modulus of continuity as $\Delta \to 0$.
\end{lemma}

\begin{proof}
By definition, we have
\[
\S^T_2(\cS^\Delta_{t,\#}\mup; r)^2
=
\int_{L^2_x} \left( 
\int_0^T
\S_2(\cS^\Delta_t(\bar{u});r)^2 \, dt 
\right) \, \mup(d\bar{u}),
\]
where, setting $u^\Delta = \cS^\Delta_t(\bar{u})$, 
\begin{align*}
\S_2(u^\Delta;r)^2 
&= \int_{\T^d} \fint_{B_r(0)} |u^\Delta(x+h)-u^\Delta(x)|^2 \, dh \, dx \\
&= \fint_{B_r(0)} \Vert u^\Delta(\slot + h) - u^\Delta(\slot) \Vert_{L^2_x}^2 \, dh \\
&\le \fint_{B_r(0)} \Vert \nabla u^\Delta \Vert_{L^2_x}^2 \, |h|^2 \, dh \\
&\le \Vert \nabla u^\Delta \Vert_{L^2_x}^2 r^2.
\end{align*}
By \ref{ass:sv2}, with $u^\Delta(t) = \cS^\Delta_t(\bar{u})$, we have 
\[
\int_0^T \Vert \nabla u^\Delta \Vert_{L^2_x}^2 \, dt
\le 
\frac{1}{2\nu} \Vert \bar{u} \Vert_{L^2_x}^2.
\]
Thus, from the above estimates, we conclude that 
\[
\int_0^T \S_2(\cS^\Delta_t(\bar{u}); r)^2 \, dt
= \int_0^T \S_2(u^\Delta; r)^2 \, dt
\le r^2 \int_0^T \Vert \nabla u^\Delta \Vert_{L^2_x}^2 \, dt
\le \frac{r^2}{2\nu} \Vert \bar{u} \Vert_{L^2_x}^2.
\]
Integration against $\mup(d\bar{u})$ yields
\begin{align*}
\S_2^T(\cS^\Delta_{t,\#}\mup;r)^2
\le 
\frac{r^2}{2\nu} \int_{L^2_x} \Vert \bar{u} \Vert_{L^2_x}^2 \, \mup(d\bar{u}),
\end{align*}
as claimed.
\end{proof}

Conditional on Assumption \ref{ass:pushforward}, we can now prove a compactness result for the filtering problem.

\begin{lemma} \label{lem:filter-compactness}
Let $\cS^\Delta_t: L^2_x \to L^2_x$ be a sequence of approximate solution operators satisfying \ref{ass:sv1p}--\ref{ass:sv3p}. Let $\nu^{\Delta,\bm{y}}_t$ be the solution of the associated filtering problem with prior $\mup \in \P_1(L^2_x)$ having bounded support, and measurements $\bm{y} = (y_1, \dots, y_N)$. If assumption \ref{ass:pushforward} holds, then $\nu^{\Delta,\bm{y}}_t$ is a compact sequence in $L^1_t(\P)$, as $\Delta \to 0$.
\end{lemma}

\begin{proof}
We observe that the structure function can be written as
\begin{align*}
\S^T_2
\left(
\nu^{\Delta,\bm{y}}_t;r
\right)^2
&=
\int_0^T \int_{L^2_x} 
\S_2(u;r)^2 
\, \nu^{\Delta,\bm{y}}_t(du) \, dt
\\
&=
\sum_{k=1}^N
\int_{t_{k-1}}^{t_k} \int_{L^2_x} 
\S_2(u;r)^2 
\, \nu^{\Delta,\bm{y}}_t(du) \, dt
\\
&=
\sum_{k=1}^N
\int_{t_{k-1}}^{t_k} \int_{L^2_x} 
\S_2(u;r)^2 
\, \left[\cS^\Delta_{t,\#}\mu^{\Delta,y_{1:(k-1)}}\right](du) \, dt
\\
&=
\sum_{k=1}^N
\int_{t_{k-1}}^{t_k} \int_{L^2_x} 
\S_2(\cS^\Delta_t(\bar{u});r)^2 
\, \mu^{\Delta,y_{1:(k-1)}}(d\bar{u}) \, dt
.
\end{align*}
We recall that on the last line, $\mu^{\Delta,y_{1:(k-1)}}$ is the solution of the BIP \eqref{eq:bipsol}. Using the uniform boundedness result for such BIP, Lemma \ref{lem:Zbound}, we conclude that
\begin{align*}
\frac{d\mu^{\Delta,y_{1:(k-1)}}}{d\mup}
&\le 
C\exp
\left(
\sum_{j=1}^{k-1}
|y_{j}|_{\Gamma}^2 + 
\sum_{j=1}^{k-1}\Vert \L^\Delta_{j}(\bar{u}) \Vert_{L^2(\mup)}^2
\right)
\\
&\le 
C\exp
\left(
|\bm{y}|_{\Gamma}^2 + \sum_{j=1}^N \Vert \L^\Delta_{j}(\bar{u}) \Vert_{L^2(\mup)}^2
\right).
\end{align*}
By our boundedness assumption \ref{ass:sv1p}, and the upper bound on observables (cp. equation~\eqref{eq:Llip}), it follows that there exists a constant $C = C(\Lip(\cG_j),B)>0$, such that
\begin{align*}
\Vert \L^\Delta_{j} (\bar{u})\Vert_{L^2(\mup)}
=
C\left( 
1 + 
\Vert 
\bar{u}
\Vert_{L^2(\mup)}
\right)
< \infty,
\end{align*}
is uniformly bounded. Since $\bm{y}$ is fixed, we conclude that there exists a constant $C> 0$, such that 
\[
\frac{d\mu^{\Delta,y_{1:(k-1)}}}{d\mup} \le C,
\]
and hence 
\begin{align*}
\S_2^T
\left(
\nu^{\Delta,\bm{y}}_t;r
\right)^2
&=
\sum_{k=1}^N
\int_{t_{k-1}}^{t_k} \int_{L^2_x} 
\S_2(\cS^\Delta_t(\bar{u});r)^2 
\, \mu^{\Delta,y_{1:(k-1)}}(d\bar{u}) \, dt
\\
&\le 
\sum_{k=1}^N
\int_{t_{k-1}}^{t_k} \int_{L^2_x}
\S_2(\cS^\Delta_t(\bar{u});r)^2 
\,C \mup(d\bar{u}) \, dt
\\
&=
C 
\sum_{k=1}^N
\int_{t_{k-1}}^{t_k} \int_{L^2_x} 
\S_2(u;r)^2 
\, \left[\cS^\Delta_{t,\#}\mup\right](du) \, dt
\\
&=
C \int_0^T \int_{L^2_x} \S_2(u;r)^2 \, \left[\cS^\Delta_{t,\#}\mup\right](du) \, dt
\\
&=
C\S_2^T
\left(\cS^\Delta_{t,\#}\mup;r\right)^2
\\
&\le C\phi(r)^2.
\end{align*}
The last estimate follows from assumption \ref{ass:pushforward}. Thus, $\S_2^T(\nu^{\Delta,\bm{y}}_t;r)\le \sqrt{C}\phi(r)$ is uniformly bounded by a modulus of continuity, implying compactness in $L^1_t(\P)$, by Proposition \ref{prop:cpct}.
\end{proof}
 
Combining the uniform stability result of Theorem \ref{thm:filtering-wp} (applied to the numerical approximations $\cS^\Delta_t$) with  Lemma \ref{lem:filter-compactness} and the a priori structure function estimate of Lemma \ref{lem:coerc}, we can now prove Theorem \ref{thm:filtering} on the compactness of the approximate filtering distributions for the Navier-Stokes equations.

\begin{proof}[Proof of Theorem \ref{thm:filtering}]
\label{pf:filtering}
We first note that by \ref{ass:sv1p}, there exists a constant $\bar{B}_\cS>0$, such that \emph{uniformly in $\Delta>0$}, we have $\Vert \cS^\Delta_t(\bar{u}) \Vert_{L^2_x} \le \bar{B}_\cS \Vert \bar{u} \Vert_{L^2_x}$, for all $\bar{u}\in L^2_x$. By Theorem \ref{thm:filtering-wp}, applied to the forward operator $\cS^\Delta_t$, the mapping
\[
\R^{d\times N}
\ni
\bm{y}
\mapsto 
\nu^{\Delta,\bm{y}}_t 
\in L^1_t(\P),
\]
is uniformly bounded on any compact subsets $K\subset \R^{d\times N}$ and uniformly equicontinuous on $K$; indeed, there exists $R>0$, such that any $\bm{y}\in K$ satisfies $|\bm{y}|_\Gamma\le R$. By Theorem \ref{thm:filtering-wp}, there exists a constant $C = C(R,\rho,N,\bar{B}_\cS, \Vert \bar{u} \Vert_{L^1(\mup)})>0$, independent of $\Delta > 0$, such that for any $t\in [0,T]$, we have
\begin{align} \label{eq:W1000}
W_1(\nu^{\Delta,\bm{y}}_t, \nu^{\Delta,\bm{y}'}_t)
\le 
C |\bm{y} - \bm{y}'|_\Gamma,
\quad
\forall |\bm{y}|_\Gamma, |\bm{y}'|_\Gamma \le R.
\end{align}
We note that in the present case, under the assumptions of Theorem \ref{thm:filtering}, we actually have $\bar{B}_\cS = 1$, and $\Vert \bar{u} \Vert_{L^1(\mup)} \le M$. Hence the constant $C = C(R,\rho,N,M)$ depends only on $R$, the noise distribution $\rho$, the number of measurements $N$ and on the upper bound on the support $M$. An immediate consequence of \eqref{eq:W1000} is that 
\[
\int_0^T W_1(\nu^{\Delta,\bm{y}}_t, \nu^{\Delta,\bm{y}'}_t) \, dt \le CT |\bm{y} - \bm{y}'|_\Gamma, \quad \forall \, |\bm{y}|_\Gamma, |\bm{y}'|_\Gamma \le R.
\]
Furthermore, by Lemma \ref{lem:filter-compactness}, the sets
\[
\set{
\nu^{\Delta,\bm{y}}_t
}{
\Delta > 0
}
\subset
L^1_t(\P),
\]
are pre-compact for any fixed $\bm{y}\in \R^{d\times N}$ (pointwise compactness). By the Arzel\'a-Ascoli theorem \ref{thm:arzelaascoli}, the claimed compactness result follows.
\end{proof}

%\begin{remark}
%In practice, a very popular choice of priors are Gaussian priors $\mup \sim \Normal(m,\Gamma)$ on function spaces, i.e. priors $\mup$ such that each finite-dimensional projection is Gaussian. We point out in passing that Theorems \ref{thm:filtering-wp}, \ref{thm:filtering-consistency} and \ref{thm:filtering} on the stability and compactness properties of approximate posteriors apply in particular when the prior is Gaussian.
%\end{remark}

%As a consequence of results of the previous sections and the interpretation of BIP as a special case of Bayesian data assimilation (cp. Remark \ref{rmk:bip-as-da}), we can prove the corresponding compactness for approximations to the Bayesian inverse problem, stated in Theorem \ref{thm:bipcompact} on page \pageref{thm:bipcompact}:
%
%\begin{proof}[Proof of Theorem \ref{thm:bipcompact}]
%{\color{red} [fill in the proof!]}
%\end{proof}

We finally would like to emphasize that while we have chosen the incompressible Navier-Stokes equations as our main prototypical example of ill-posed problems arising in fluid mechanics, the results of the present section apply to more general models, as will be briefly indicated next.

\subsection{Comment on related models: Hyperbolic systems of conservation laws}

The results of this work apply, for example, to the numerical approximation of Bayesian inverse problems for hyperbolic systems of conservation laws. Again, we take as our domain $D = \T^n = [0,2\pi]^n$ with periodic boundary conditions.  We recall that a system of conservation laws is a PDE of the form 
\begin{align} \label{eq:hyp}
\partial_t u^i + \sum_{j=1}^n \partial_j \left( F^{ij}(u)\right) &= 0,
\end{align}
describing the temporal evolution of $m$ conserved quantities $u^1, \dots, u^m: D \times [0,T] \to \R$, and $F^{ij}: \R^m \to \R$ are the fluxes. It is convenient to write the system \eqref{eq:hyp} in the succinct form 
\[
\partial_t u + \div(F(u)) = 0,
\]
where $u=(u^1, \dots, u^m): D \times [0,T] \to \R^m$ and $F = (F^{ij}): \R^m \to \R^{m\times n}$. 
The system of conservation laws \eqref{eq:hyp} is called hyperbolic, provided that the Jacobian $D_u(F \cdot n)$ possesses real eigenvalues for all unit vectors $n\in \R^n$ with $|n|=1$. 
A great variety of systems in continuum mechanics can be formulated as hyperbolic systems of conservation laws, including the compressible Euler equations of gas dynamics, the shallow water equations of oceanography, the Magneto-Hydro-Dynamics (MHD) equations of plasma physics, and the equations of nonlinear elastodynamics \cite{Dafermos2005}.

As it is well-known, even in the special case of a scalar conservation law ($m =1$), weak solutions to \eqref{eq:hyp} are not necessarily unique. It is therefore necessary to augment hyperbolic conservation laws \eqref{eq:hyp} with additional entropy, or admissibility conditions. These entropy conditions are based on the existence of an entropy/entropy-flux pair $(\eta,q)$ consisting of a convex function $\eta: \R^m \to \R$ and a flux $q: \R^m \to \R^n$, such that 
\[
D_u q = D_u \eta \cdot D_u F.
\]
Here, $D_u q$, $D_u \eta$ denote the Jacobian matrices of $q(u)$ and $\eta(u)$. A weak solution $u$ of \eqref{eq:hyp} is called an \emph{entropy} weak solution, provided that, in addition to \eqref{eq:hyp}, also
\begin{align} \label{eq:entropy}
\partial_t \eta(u) + \div(q(u)) \le 0,
\end{align}
holds in the sense of distributions.

In the following we will restrict our attention to hyperbolic systems of conservation laws for which 
\[
\Vert D^2 F\Vert_{L^\infty} < \infty,
\]
and which admit a coercive, smooth flux function $\eta(u)\in C^2(\R^m)$, in the sense that there exist constants $c,C>0$, such that
\[
c 
\le
\left(
D^2 \eta(u) v, v\right)
\le
C, 
\quad 
\forall u, v\in \R^m, \; \text{with }|v|=1.
\]
Note that in this case, the entropy admissibility condition \eqref{eq:entropy} implies, upon integration over $D$, an a priori bound of the form
\[
\Vert u(t) \Vert_{L^2} \le B_\cS \Vert \overline{u} \Vert_{L^2},
\]
for any admissible weak solution $u$ with initial data $u(t=0) = \overline{u}$.

\subsubsection{Numerical methods}
In the context of systems of conservation laws, a popular method of choice are finite volume and finite difference methods, as e.g. employed in the numerical experiments for statistical solutions of \cite{FLMW1}. We briefly review the form of these numerical schemes, following \cite[Section 4.1]{FLMW1}. For a more complete review, we refer to e.g. \cite{GRbook,Leveque}.

The computational spatial domain is discretized by a collection of cells 
\[
\{
(x^1_{i^1-1/2},x^1_{i^1+1/2}) 
\times \dots \times 
(x^n_{i^n-1/2},x^n_{i^n+1/2}) 
\}_{(i^1, \dots, i^n)},
\]
with corresponding cell midpoints
\[
x_{i^1,\dots, i^n}
=
\left(
\frac{x^1_{i^1+1/2}+x^1_{i^1-1/2}}{2},
\dots,
\frac{x^n_{i^n+1/2}+x^n_{i^n-1/2}}{2},
\right).
\]
We assume that the mesh is equidistant, i.e. for some $\Delta > 0$ we have 
\[
x^k_{i^k+1/2} - x^k_{i^k-1/2} \equiv \Delta, \quad \forall \, k=1,\dots,n.
\]
For $\i = (i^1,\dots, i^n)$, we denote the averaged value in the cell at time $t\ge 0$ by $u^\Delta_{\i}(t) = u^\Delta_{(i^1,\dots, i^n)}(t)$. We consider the following semi-discrete scheme
\begin{align} \label{eq:FVscheme}
\frac{d}{dt} u^\Delta_{i^1,\dots,i^n}(t)
+
\sum_{k=1}^n
\frac{1}{\Delta}
&\left(
F^{k,\Delta}(u^\Delta_{\i-(q-1)\e_k}(t),\dots, u^{\Delta}_{\i+q\e_k}(t))
\right.
\\
&
\left.
-
F^{k,\Delta}(u^\Delta_{\i-q\e_k}(t),\dots, u^{\Delta}_{\i+(q-1)\e_k}(t))
\right)
=0,
\end{align}
and $u^\Delta_{i^1,\dots,i^n}(0) = \overline{u}(x_{i^1,\dots, i^n})$. Here, $\e_1,\dots, \e_n$ are the canonical unit vectors in $\R^n$. $F^{k,\Delta}$ denotes the numerical flux function in direction $k=1,\dots, n$, and $\overline{u}_\i = \fint_{C_\i} u_0(x) \, dx$ is the average of the initial data over the $\i$-th cell.

As in \cite{FLMW1}, we make the following assumptions on the discretization \eqref{eq:FVscheme}:

\begin{assumption} \label{ass:FV}
We assume that the finite volume scheme \eqref{eq:FVscheme} is consistent in the sense that there exists a constant $C>0$ such that for $k=1,\dots, d$,
\[
|F^{k,\Delta}(u^\Delta_{\i-(q-1)\e_k},
\dots, 
u^{\Delta}_{\i+q\e_k})
-
f^k(u^\Delta_\i)
|
\le 
C \sum_{j=-q+1}^q
|u^\Delta_\i(t) - u^\Delta_{\i+j\e_k}(t)|
\]
and the discretized solutions satisfy
\begin{enumerate}
\item $L^2$ bound: There exists $C>0$ such that
\[
\Delta^d \sum_{\i} |u^\Delta_\i(t)|^2
\le 
C \Delta^d \sum_{\i} |\overline{u}_\i|^2,
\]
\item weak BV bound: There exists $s\ge 2$, such that 
\[
\Delta^d \int_0^T \sum_{k=1}^d
\sum_{\i} |u^\Delta_{\i+\e_k}(t) - u^\Delta_\i(t)|^s \, dt \le C\Delta,
\]
with the constant $C = C(\Vert \overline{u} \Vert_{L^2_x})$ depending only on the $L^2$-norm of the initial data. 
\end{enumerate}
\end{assumption}

\begin{remark}
It is not difficult to see that, under Assumption \ref{ass:FV}, the discrete solution operator $S^\Delta_t: L^2_x \to L^2_x$, which is obtained by locally constant reconstruction (or suitable higher-order variants),
\[
S^\Delta_t(\overline{u})
:=
\sum_{\i} u^\Delta_{\i}(t) 1_{C_\i},
\]
satisfies the assumptions \ref{ass:sv1p} and \ref{ass:sv3p} of Section~\ref{sec:bda-compact}. Furthermore, as pointed out in \cite[Remark 4.2]{FLMW1}, many examples of finite volume/difference schemes can be shown to satisfy Assumption \ref{ass:FV}. Examples include the so-called entropy stable Lax-Wendroff schemes and the TeCNO schemes of \cite{Fjordholm2012}.
\end{remark}

Based on the results of sections~\ref{sec:BIP} and~\ref{sec:bda-compact}, we immediately obtain:

\begin{result}
Assume that the FV scheme \eqref{eq:FVscheme} satisfies Assumption~\ref{ass:FV}. If $\mup \in \P_1(L^2_x)$, then the posteriors $\mu^{\Delta,y}$ of the BIP \eqref{eq:posterior} for the hyperbolic conservation law are uniformly stable in $y$, for any $\Delta > 0$; i.e., for any $R>0$ there exists a constant $C>0$, independent of $\Delta$, such that 
\[
W_1(\mu^{\Delta,y},\mu^{\Delta,y'}) \le C |y-y'|_\Gamma,
\quad \forall |y|_\Gamma,|y'|_\Gamma\le R.
\]
Furthermore, the posteriors $\mu^{\Delta,y}$ form a compact sequence in $\P_1(L^2_x)$.
\end{result}

This result complements recent work in \cite{mishra2021wellposedness}, where (uncondintional) convergence was shown for the case of scalar conservation laws ($m=1$). For the filtering problem, we have:

\begin{result}
Assume that the FV scheme \eqref{eq:FVscheme} satisfies Assumption~\ref{ass:FV}. Then the approximate solutions $\nu^{\Delta,\bm{y}}_t$ of the filtering problem computed by the FV scheme  are uniformly stable with respect to the measurements $\bm{y}$, in the sense of \eqref{eq:filtering-stability}, for any $\Delta > 0$. In addition, if the prior $\mup \in \P_1(L^2_x)$ satisfies Assumption~\ref{ass:pushforward}, then the posteriors $\nu^{\Delta,\bm{y}}_t$ form a compact sequence in $L^1_t(\P)$.
\end{result}

\begin{remark}
The validity of Assumption~\ref{ass:pushforward} has been investigated for a diverse set of initial priors in \cite{FLMW1}. The numerical evidence, presented in \cite{FLMW1} strongly suggest that it is fulfilled for the cases considered there, and we conjecture that the structure functions \eqref{eq:pushforward} are uniformly bounded for a wide range of priors of practical relevance.
\end{remark}

\begin{remark}
The current section has been formulated for hyperbolic systems of convergence laws with a strictly convex entropy. The main reason for this restriction is that the results of Section~\ref{sec:bda-compact}, and the compactness proof of \cite{LMP1} are based on the $L^2$-framework that is natural in the context of the incompressible Euler equations. However, there should be no essential difficulty in extending these results to $L^p$-based spaces for $p\ne 2$.
\end{remark}

\section{Discussion}
\label{sec:discussion}

The Bayesian framework has been well-established as a suitable formulation of inverse problems
%and data assimilation  
arising in the context of PDEs \eqref{eq:pde} \cite{Stuart2010}. The well-posedness of the Bayesian inverse problem has been demonstrated under very mild assumptions on the well-posedness of the \emph{forward problem} for the underlying PDE, requiring essentially only the existence and uniqueness of solutions defined on an infinite-dimensional Banach space \cite{Latz2020,Sprungk2020}. Corresponding well-posedness results for data assimilation have focused mostly on finite-dimensional problems with Gaussian noise and when the solution operator is continuous \cite{law2015data}.

However, for a large numbers of PDEs, such as the fundamental equations of fluid dynamics, the forward problem may not be well-posed. Existence, uniqueness or stability of solutions are either not true or can not be established rigorously. This issue is further exacerbated by the fact that for many of these PDEs, numerical approximations either may not converge on mesh refinement or converge too slowly to be useful. This is often a result of the sensitivity of solutions to small perturbations and the appearances of structures at smaller and smaller scales, as the grid is refined \cite{FKMT2017,FLMW1,Glimm2001}. 

Our main aim in this paper was to investigate Bayesian data assimilation (filtering) for such PDEs with a very unstable or even \emph{ill-posed forward problem}. Our main results, summarized in Section \ref{sec:main}, concern the properties of the time-dependent filtering distribution $\nu^{\bm{y}}_t$ (exact posterior based on the underlying "ground truth" map $\cS^\dagger_t$) and numerical approximations $\nu^{\Delta,\bm{y}}_t$ based on a approximate solution operator $\cS^\Delta_t$, with $\Delta > 0$ a discretization parameter (grid size) and $\bm{y} = (y_1,\dots,y_N)$ being finite-dimensional (noisy) measurements acquired over time $t\in [0,T]$. We were able to show: 
\begin{itemize}
\item \textbf{(Well-posedness)} We prove that Bayesian filtering is \emph{well-posed} under very general assumptions on the forward solution operator $\cS^\dagger_t: X \to X$; in particular, we show that the measurement-to-posterior mapping $\bm{y} \mapsto \nu^{\bm{y}}_t$ is locally Lipschitz continuous, even if the forward mapping $\bar{u} \mapsto \cS^\dagger_t(\bar{u})$ is discontinuous (cp. Theorem \ref{thm:filtering-wp}).
\item \textbf{(Consistency)} We prove that approximations $\nu^{\Delta,\bm{y}}_t \approx \nu^{\bm{y}}_t$ of the filtering distribution, e.g., arising from numerical discretization of the underlying PDE at mesh size $\Delta > 0$, \emph{converge} to the exact posterior as $\Delta \to 0$, provided that the approximate solution operators $\cS^\Delta_t \to \cS^\dagger_t$ converge only in a mean-square sense relative to the prior, (cp. Theorem \ref{thm:filtering-consistency}).
\item \textbf{(Compactness and stability)} We demonstrate that even in the absence of any convergence-guarantees of the approximate solution operators $\cS^\Delta_t: X \to X$, the corresponding approximate Bayesian filtering distributions $\{\mu^{\Delta,y}\}_{\Delta>0}$ form a \emph{compact sequence} under mild conditions (satisfied e.g. by the Navier-Stokes equations), and hence possess limit points, (cp. Theorem \ref{thm:filtering}).
\end{itemize}
The well-posedness results in the context of Bayesian data assimilation presented in this work, even for models for which the forward problem may be ill-posed, have been derived under mild assumptions and are applicable to a wide range of models encountered in practice. The stability results should be of particular significance to practitioners, as they demonstrate that under readily verifiable conditions on the numerical scheme, the approximate solutions of the data assimilation problem are stable with respect to perturbations of the measurements, \emph{independently of the numerical resolution and physical parameters such as the viscosity}. 

Our consistency results do not only imply the convergence of the filtering distributions $\nu^{\Delta,\bm{y}}_t \to \nu^{\bm{y}}_t$ given convergent approximations of the forward problem $\cS^\Delta_t \to \cS^\dagger_t$, but they also provide \emph{quantitative error bounds} on the (Wasserstein-) distance between the exact and approximate filtering distributions. These upper bounds are obtained in terms of the mean-square distance between the exact and approximate solution operators with respect to the prior. Such quantitative estimates are not only of importance in studying the convergence of Bayesian filtering based on traditional numerical discretizations, but also open up the possibility of deriving similar quantitative bounds for Bayesian data assimilation based on novel neural network-based operator learning frameworks such as \cite{fourierop2020}, extending the work \cite{FNObounds} to the Bayesian filtering context (cp. Remark \ref{rmk:surrogate}). 

Finally, the general compactness properties uncovered in the present work allow us to define a set of candidate solutions to the filtering problem, generated by suitable numerical schemes. As this set can be shown to be non-empty a priori, this potentially opens up the possibility of identifying the correct solution among these candidates by a suitable selection criterion to single out a ``canonical'' posterior amongst the set of candidate solutions. We propose to further investigate these questions in forthcoming work.

\appendix

\section{Mathematical complements}

\subsection{\texorpdfstring{$L^p$}-norms}
\label{sec:Lp}

Let $X$ be a separable Banach space, and let $\mup\in \cP(X)$ be a probability measure. We introduce the $L^p(\mup)$-norm ($p\in [1,\infty]$) of a Borel measurable mapping $\cF: X \to Y$, $\bar{u} \mapsto \cF(\bar{u})$, with $X$, $Y$ Banach space, as follows:
\begin{align}
\label{eq:L2mu}
\Vert \cF \Vert_{L^p({\mup})}
:=
\begin{cases}
\displaystyle{\left(
\int_{X}
\Vert \cF(\bar{u}) \Vert_Y^p
\, 
{\mup}(d\bar{u})
\right)^{1/p}
}, & (p<\infty), \\
\displaystyle{
\esssup_{\bar{u} \sim \mup} \Vert \cF(\bar{u}) \Vert_{Y},
}
& (p=\infty),
\end{cases}
\end{align}
where we recall that for $p=\infty$, the essential supremum is defined by
\[
\esssup_{\bar{u}\sim \mup} \Vert \cF(\bar{u}) \Vert_{Y}
:=
\sup
\Big\{M>0 \, \Big| \,  \mup(\{\bar{u} \, |\, \Vert \cF(\bar{u}) \Vert_{X} >M\}) > 0\Big\}.
\]
In particular, if $Y=X$ and if $X\to X$, $\bar{u} \mapsto \cF(\bar{u}) = \bar{u}$ is given by the identity mapping, then we have (for $p<\infty$):
\[
\Vert \bar{u} \Vert_{L^p(\mup)}
=
\left(
\int_{X} \Vert \bar{u} \Vert_{X}^p \, \mup(d\bar{u})
\right)^{1/p}.
\]

\subsection{Wasserstein distance}
\label{sec:Wasserstein}
In this section, we introduce the notation for the rest of the paper and recall some preliminaries that are necessary to define the Bayesian inverse problem in a mathematically precise manner. 

Given a separable Hilbert space $X$, we denote by $\P(X)$ the space of Borel probability measures on $X$. The term ``measurable'' will always refer to Borel measurability. A sequence $\mu^\Delta \in \P(X)$ is said to converge weakly to a limit $\mu$, denoted $\mu^\Delta \weaklyto \mu$, if 
\[
\int_X \phi(u) \, \mu^\Delta(du) \to \int_X \phi(u) \, \mu(du), 
\qquad
\forall \phi \in C_b(X),
\]
where $C_b(X)$ denotes the space of bounded, continuous functions on $X$.
We denote by $\P_p(X)$ the space of Borel probability measures $\mu \in \P(X)$, possessing finite $p$-th moments, $\int_X \Vert u \Vert_X^p \, \mu(du) < \infty$, metrized by the $p$-Wasserstein distance $W_p$:
\begin{align}
\label{eq:Wp}
W_p(\mu,\nu)
:=
\inf_{\pi\in \Gamma(\mu,\nu)}
\left(
\int_{X\times X} \Vert u-v \Vert_X^p \, \pi(du,dv)
\right)^{1/p}.
\end{align}
Here, $\Gamma(\mu,\nu)$ is the set of couplings between $\mu$ and $\nu$, i.e. probability measures $\pi$ on $X\times X$, with projections $(\proj_1)_\# \pi = \mu$, $(\proj_2)_\# \pi = \nu$. Given a map $F: X\to Y$, we denote by $F_\#\mu \in \P(Y)$ the push-forward of a probability measure $\mu \in \P(X)$ by $F$; the push-forward measure satisfies the relation 
\[
\int_Y \phi(v) \, \left(F_\#\mu\right)(dv)
=
\int_X (\phi \circ F)(u) \, \mu(du),
\]
for all measurable functions $\phi: Y \to \R$ such that $\phi \circ F \in L^1(\mu)$. We recall that the $1$-Wasserstein distance $W_1(\mu,\nu)$ between measures $\mu,\nu\in \P_1(X)$ can also be determined via the Kantorovich duality:
\begin{align} \label{eq:kantorovich}
W_1(\mu,\nu)
=
\sup_{\Phi}
\int_X \Phi(u) \left[\mu(du) - \nu(du)\right],
\end{align}
where the supremum is taken over all Lipschitz continuous $\Phi \in \Lip(X)$, with $\Vert \Phi \Vert_{\Lip} \le 1$, and we define the semi-norm $\Vert \slot \Vert_{\Lip}$ by 
\begin{align} \label{eq:Lip}
\Vert \Phi \Vert_{\Lip} 
:=
\sup_{u\ne v} \frac{|\Phi(u) - \Phi(v)|}{\Vert u - v\Vert_X}.
\end{align}
We also recall that for a sequence of measures $\mu^{\Delta}\in \P_1(X)$, $\Delta \to 0$, and $\mu \in \P_1(X)$, we have
\begin{align}\label{eq:W1conv}
\lim_{\Delta \to 0} W_1(\mu^\Delta,\mu) = 0
\iff
\left\{
\begin{gathered}
\mu^\Delta \weaklyto \mu \text{ converges weakly and}
\\
\int_X \Vert u \Vert_X \, \mu^{\Delta}(du)
\to 
\int_X \Vert u \Vert_X \, \mu(du).
\end{gathered}
\right\}
\end{align}
We finally prove that if $\cF: X \to X$ is a Lipschitz continuous map and $\cF_\#: \cP_1(X) \to \cP_1(X)$ denotes the push-forward under $\cF$, then $\Lip(\cF:X \to X) = \Lip(\cF_\#: \cP_1(X) \to \cP_1(X))$, which implies the claim of Proposition \ref{prop:LipW1}, for $\cF = \cS^\dagger_t$.

\begin{proof}[Proof of Proposition \ref{prop:LipW1}]
\label{pf:LipW1}
Let $\cF: X \to X$ be any Lipschitz continuous map (in particular, the following applies to $\cF = \cS^\dagger_t: X \to X$).
By definition, we have 
\begin{align*}
W_1(\cF_\#\mu, \cF_\#\mu')
&=
\inf_{\pi \in \Gamma(\cF_\#\mu,\cF_\#\mu')}
\int_{X\times X} \Vert u - u' \Vert_X \, \pi(du,du')
\\
&\le
\inf_{ \pi \in \Gamma(\mu,\mu')}
\int_{X\times X} \Vert u - u' \Vert_X \, (\cF\times \cF)_\#\pi(du,du')
\\
&=
\inf_{ \pi \in \Gamma(\mu,\mu')}
\int_{X\times X} \Vert \cF(u) - \cF(u') \Vert_X \, \pi(du,du')
\\
&\le 
\Lip(\cF) \, 
\inf_{ \pi \in \Gamma(\mu,\mu')}
\int_{X\times X} \Vert u - u' \Vert_X \, \pi(du,du')
\\
&=
\Lip(\cF) \, W_1(\mu,\mu').
\end{align*}
To see the optimality of $\Lip(\cF)$, we note that for any $\tilde{L}>0$, such that $W_1(\cF_\#\mu,\cF_\#\mu') \le \tilde{L} W_1(\mu,\mu')$ holds for all $\mu,\mu'$, we have
\begin{align*}
\tilde{L} 
&\ge 
\sup_{\mu\ne \mu' \in \cP_1(X)} \frac{W_1(\cF_\#\mu, \cF_\#\mu')}{W_1(\mu,\mu')}
\ge 
\sup_{u\ne u' \in X} \frac{W_1(\cF_\#\delta_u, \cF_\#\delta_{u'})}{W_1(\delta_u,\delta_{u'})}
\\
&=
\sup_{u\ne u' \in X} \frac{\Vert \cF(u) - \cF(u') \Vert_X}{\Vert u - u'\Vert_X}
=
\Lip(\cF).
\end{align*}
\end{proof}

%We will denote the Kullback-Leibler (KL) divergence of a measure $\nu\in \P(X)$ with respect to $\mu\in \P(X)$ by $\KL(\nu || \mu)$; We recall that the Kullback-Leibler divergence is defined by
%\begin{align} \label{eq:KL}
%\KL(\nu||\mu)
%:=
%\begin{cases}
%\int_X \log\left(\frac{d\nu}{d\mu}\right) \, d\nu, & (\nu \ll \mu),
%\\
%+\infty, & (\nu \not \ll \mu).
%\end{cases}
%\end{align}
%It is well-known that $\P(X) \to \R$, $\nu \mapsto \KL(\nu||\mu)$ is a strictly convex, coercive and lower semi-continuous function. In particular, for any $\alpha > 0$ the set $\set{\nu \in \P(X)}{\KL(\nu||\mu) \le \alpha}$ is compact in the weak topology on $\P(X)$.

\subsection{Compactness}
We recall the Arzela-Ascoli theorem, characterizing compactness in $C_\mathrm{loc}(X,Y)$:
\begin{theorem}[Arzela-Ascoli] \label{thm:arzelaascoli}
Let $X$ be a locally compact Hausdorff space. Let $Y$ be a complete metric space. A subset $F\subset C_\loc(X,Y)$ is relatively compact iff it is equi-continuous and for all $x\in X$, the set $\set{f(x)}{f \in F}$ is relatively compact in $Y$.
\end{theorem}

%%%%%%%%%%%%%%%%%%%%%%%%%%%%%%%%%%
%%% The following is not essential for the discussion...
%%%%%%%%%%%%%%%%%%%%%%%%%%%%%%%%%%

\if{0}
The following theorem relates the compactness of a family of measures $\mu^{\Delta}_t \in L^1_t(\P)$, where $L^1_t(\P) = L^1_t([0,T];\P(L^2_x(\T^d)))$ is metrized by the distance $d_T(\mu_t,\nu_t)$, 
\[
d_T(\mu_t, \nu_t)
:=
\int_0^T W_1(\mu_t,\nu_t) \, dt, \quad \text{($W_1$ = $1$-Wasserstein distance)},
\]
to the uniform decay of their structure functions
\[
\S^T_2(\mu_t^\Delta;r)
:=
\left(
\int_0^T \int_{L^2_x}
\int_{\T^d} \fint_{B_r(0)}
|u(x+h)-u(x)|^2 
\, dh \, dx \, \mu_t(du) \, dt
\right)^{1/2}.
\]

\begin{proposition} \label{prop:L1tP-cpct}
Let $\cS^\Delta: [0,T] \times L^2_x \to L^2_x$, be a family of operators, such that $(t,u) \mapsto \cS^\Delta_t(u)$ is continuous for each $\Delta > 0$. Assume that there exists a constant $\bar{B}_\cS>0$, such that $\Vert \cS^\Delta_t(u)\Vert_{L^2_x} \le \bar{B}_\cS \Vert u \Vert_{L^2_x}$ for all $t\in [0,T]$ and $\Delta > 0$. Let $\mup\in \P(L^2_x)$ be a probability measure on $L^2_x(\T^d)$ with finite second moments, i.e. such that
\[
\int_{L^2_x} \Vert \bar{u} \Vert_{L^2_x}^2 \, \mup(d\bar{u}) < \infty,
\]
and define a family of probability measures $\{\mu^\Delta_t\}_{\Delta>0}$ as the push-forward of $\mup$ under $\cS^\Delta_t$, i.e. $\mu^\Delta_t := \cS^\Delta_{t,\#} \mup \in L^1_t(\P)$. If $\{\mu^\Delta_t\}_{\Delta > 0}$ is relatively compact in $L^1_t(\P)$, then there exists a modulus of continuity $\phi: [0,1]\to [0,\infty)$, $r\mapsto \phi(r)$, such that 
\[
\S^T_2(\mu^{\Delta}_t; r) \le \phi(r), \quad \forall \, \Delta > 0, \; \forall \, r\in [0,1].
\]
\end{proposition}

\begin{proof}
To prove the claim, it suffices to show that $\limsup_{r\to 0} \sup_{\Delta > 0} \S^T_2(\mu^{\Delta}_t;r) = 0$, i.e. that
\[
\limsup_{r\to 0} \sup_{\Delta > 0}
\left(
\int_0^T\int_{L^2_x}\int_{\T^d}\fint_{B_r(0)} |u(x+h) - u(x)|^2 \, dh \, dx \, \mu^\Delta_t(du) \, dt
\right)^{1/2}
=
0.
\]
We first note that the quantity on the left is finite for any $\Delta > 0$ and $r\in [0,1]$, since we trivially have 
\[
\S^T_2(\mu^\Delta_t;r) \le 2\left(\int_0^T \int_{L^2_x} \Vert u\Vert_{L^2_x}^2 \, \mu^\Delta_t(du) \, dt\right)^{1/2}, 
\]
and by assumption on the second moment of $\mup$ and the uniform boundedness of the operators $\cS^\Delta_t: L^2_x \to L^2_x$, there exists a constant $\bar{B}_\cS>0$, such that
\begin{align*}
\int_{L^2_x} \Vert u \Vert_{L^2_x}^2 \, \mu_t^\Delta(du)
&=
\int_{L^2_x} \Vert u \Vert_{L^2_x}^2 \, d\left(\cS^\Delta_{t,\#}\mup\right)(du)
=
\int_{L^2_x} \Vert \cS^\Delta_t(\bar{u}) \Vert_{L^2_x}^2 \, d\mup(\bar{u})
\\
&\le
\int_{L^2_x} \bar{B}_\cS^2\Vert \bar{u} \Vert_{L^2_x}^2 \, d\mup(\bar{u}) < \infty,
\end{align*}
is uniformly bounded for any $t\in [0,T]$.

To show that $\S^T_2(\mu^\Delta_t;r)$ converges to $0$, uniformly as $r\to 0$, we will use mollification. In the following, we denote by $u_\epsilon$ the $\epsilon$-mollification of $u\in L^2_x$, $u_\epsilon(x) = (\rho_\epsilon\ast u)(x)$, where $\rho_\epsilon(x) := \epsilon^{-d}\rho(x/\epsilon)$, and $\rho\ge 0$ is a smooth function supported in a ball of radius $1$, such that $\int_{B_1(0)} \rho(x) \, dx = 1$. We now note that 
\begin{align*}
\S^T_2(\mu^\Delta_t;r)
&\le
\left(\int_0^T \int_{L^2_x} \int_{\T^d}\fint_{B_r(0)} |u_\epsilon(x+h) - u_\epsilon(x)|^2 \, dh \, dx \, \mu_t(du) \, dt\right)^{1/2}
\\
&\qquad + 
\sup_{\Delta > 0} 2\left(\int_0^T \int_{L^2_x} \Vert u - u_\epsilon \Vert_{L^2_x}^2 \, \mu^\Delta_t(du) \, dt \right)^{1/2} .
\end{align*}
For any $u\in L^2_x$ and $\epsilon > 0$, we have
\[
\int_{\T^d} \fint_{B_r(0)} |u_\epsilon(x+h) - u_\epsilon(x)|^2 \, dh \, dx
\le
\Vert \nabla u_\epsilon \Vert_{L^2_x}^2 \, r^2
\le 
C \Vert u \Vert_{L^2_x}^2 \left(\frac{r}{\epsilon}\right)^2.
\]
In particular, this implies that 
\begin{align*}
\limsup_{r\to 0} \sup_{\Delta>0}
&\int_0^T \int_{L^2_x} \int_{\T^d} \fint_{B_r(0)} |u_\epsilon(x+h) - u_\epsilon(x)|^2 \, dh \, dx \, \mu^\Delta_t(du) \, dt
\\
&\le
\limsup_{r\to 0} 
C T \left(\int_{L^2_x} \Vert \bar{u} \Vert_{L^2_x}^2 \, \mup(d\bar{u})\right) \left(\frac{r}{\epsilon}\right)^2
\\
&= 0.
\end{align*}
Hence, we have
\begin{align} \label{eq:S-moll-est}
\limsup_{r\to 0} \sup_{\Delta > 0} \S^T_2(\mu^\Delta_t;r)
\le 
2\left(
\sup_{\Delta > 0}
\int_0^T \int_{L^2_x} \Vert u - u_\epsilon \Vert_{L^2_x}^2 \, \mu^\Delta_t(du) \, dt
\right)^{1/2},
\end{align}
for any $\epsilon > 0$. The statement of this proposition will thus follows from the following claim: If the family $\{\mu^\Delta_t\}_{\Delta > 0}$ is compact in $L^1_t(\P)$, then we have
\[
\limsup_{\epsilon \to 0} \sup_{\Delta > 0}
\int_0^T \int_{L^2_x} \Vert u - u_\epsilon \Vert_{L^2_x}^2 \, \mu^\Delta_t(du) \, dt
=
0.
\]
To see why, fix $M>0$ and denote $B_M = \set{u\in L^2_x}{\Vert u \Vert_{L^2_x} \le M}$. Then 
\begin{gather} \label{eq:I-II-split}
\begin{aligned}
\int_0^T \int_{L^2_x} \Vert u - u_\epsilon \Vert_{L^2_x}^2 \, \mu^\Delta_t(du) \, dt
&=
\int_0^T \int_{L^2_x\cap B_M^c} \Vert u - u_\epsilon \Vert_{L^2_x}^2 \, \mu^\Delta_t(du) \, dt
\\
&\qquad 
+
\int_0^T \int_{L^2_x\cap B_M} \Vert u - u_\epsilon \Vert_{L^2_x}^2 \, \mu^\Delta_t(du) \, dt
\\
&=:
(I^\Delta) + (II^\Delta).
\end{aligned}
\end{gather}
We can estimate
\begin{align*}
(I^\Delta) 
&\le
\int_0^T\int_{L^2_x \cap B_M^c}
\left\{
2\Vert u \Vert_{L^2_x}^2 + 2\Vert u_\epsilon\Vert_{L^2_x}^2
\right\}
\, \mu^\Delta_t(du) \, dt
\\
&\le
4\int_0^T\int_{L^2_x \cap B_M^c}
\Vert u \Vert_{L^2_x}^2
\, \mu^\Delta_t(du) \, dt
\\
&=
4\int_0^T\int_{L^2_x \cap [S^\Delta_t]^{-1}(B_M^c)}
\Vert \cS^\Delta_t(\bar{u}) \Vert_{L^2_x}^2
\, d\mup(d\bar{u}) \, dt
\end{align*}
By our assumption on the boundedness of $\cS^\Delta_t$, there exists a constant $\bar{B}_\cS>0$, such that
$\Vert \cS^\Delta_t(\bar{u}) \Vert_{L^2_x} \le \bar{B}_\cS \Vert \bar{u} \Vert_{L^2_x}$,
for all $\bar{u}\in L^2_x$, and uniformly for $\Delta > 0$. This implies that 
\[
\sup_{\Delta > 0}\;
\int_{L^2_x \cap [S^\Delta_t]^{-1}(B_M^c)}
\Vert S^\Delta_t(\bar{u}) \Vert_{L^2_x}^2
\, d\mup(d\bar{u})
\le
\int_{L^2_x \cap B_{M/\bar{B}_\cS }^c}
\bar{B}_\cS ^2\Vert \bar{u} \Vert_{L^2_x}^2
\, d\mup(d\bar{u}),
\]
and hence
\begin{align*}
\sup_{\Delta > 0} \; (I^\Delta)
&= \sup_{\Delta > 0} \int_0^T \int_{L^2_x\cap B^c_M} \Vert u - u_\epsilon \Vert^2_{L^2_x} \, \mu^\Delta_t(du) \, dt
\\
&\le
4T \bar{B}_\cS ^2 \int_{L^2_x \cap B_{M/\bar{B}_\cS }^c} \Vert \bar{u} \Vert_{L^2_x}^2 \, d\mup(d\bar{u}).
\end{align*}
Let $\delta > 0$ be arbitrary. Since $\int_{L^2_x} \Vert \bar{u} \Vert_{L^2_x}^2 \, d\mup(d\bar{u}) < \infty$, we can fix $M> 0$ sufficiently large, such that 
\begin{align}\label{eq:I-est}
\sup_{\Delta > 0} \; (I^\Delta) < \delta.
\end{align}
To estimate the second term, we note that by the assumed relative compactness of $\{\mu^\Delta_t\}\subset L^1_t(\P)$ and for the same $\delta, M > 0$ as above, there exists a finite collection $\{\mu^{\Delta_k}_t\}_{k=1,\dots, N}$, such that for any $\Delta > 0$, there exists $k \in \{1,\dots, N\}$, such that $d_T(\mu^\Delta_t, \mu^{\Delta_k}_t) < \delta/4M$. It then follows that for any index $\Delta>0$, we have
\begin{align*}
(II^\Delta)
&= \int_0^T \int_{L^2_x\cap B_M} \Vert u - u_\epsilon \Vert^2_{L^2_x} \, \mu^\Delta_t(du) \, dt
\\
&\le 
\int_0^T \int_{L^2_x\cap B_M} 2M \Vert u - u_\epsilon \Vert_{L^2_x} \, \mu^\Delta_t(du) \, dt
\\
&=
2M \int_0^T \int_{L^2_x} \Vert u - u_\epsilon \Vert_{L^2_x} \, 
\left[
\mu^\Delta_t(du) 
- \mu^{\Delta_k}_t(du)
\right]\, dt
\\
&\qquad +
2M \int_0^T \int_{L^2_x} \Vert u - u_\epsilon \Vert_{L^2_x} \, 
\mu^{\Delta_k}_t(du)
\, dt
\\
&\le
4M \int_0^T W_1(\mu^\Delta_t,\mu^{\Delta_k}_t) \, dt
+
2M \int_0^T \int_{L^2_x} \Vert u - u_\epsilon \Vert_{L^2_x} \, 
 \mu^{\Delta_k}_t(du)
\, dt
\\
&<
\delta
+
2M \int_0^T \int_{L^2_x} \Vert u - u_\epsilon \Vert_{L^2_x} \, 
 \mu^{\Delta_k}_t(du)
\, dt,
\end{align*}
for any $\epsilon > 0$. In the second to last step, we have used the $2$-Lipschitz continuity of $u \mapsto \Vert u - u_\epsilon \Vert_{L^2_x}$ and Kantorovich duality for $W_1$. In the last step, we used the fact that $\{\mu^{\Delta_k}_t\}_{k=1,\dots, N}$ is a $\delta/4M$-net for $\{\mu^\Delta_t\}\subset L^1_t(\P)$. We further note that the integrand $u \mapsto \Vert u - u_\epsilon \Vert$ converges pointwise to $0$ as $\epsilon \to 0$, and is uniformly bounded by the (integrable) function $(u \mapsto 2\Vert u \Vert_{L^2_x}) \in L^1(\mu^{\Delta_k}_t \otimes dt)$ for any $k=1,\dots, N$. By the Lebesgue dominated convergence theorem, it thus follows that 
\[
\lim_{\epsilon \to 0} \int_0^T \int_{L^2_x} \Vert u - u_\epsilon \Vert_{L^2_x} \, 
 \mu^{\Delta_k}_t(du)
\, dt = 0,
\]
for any $k=1,\dots, N$. Since the set $\{\mu^{\Delta_k}_t\}_{k=1,\dots, N}$ is finite, we conclude that
\begin{align*}
\sup_{\Delta > 0} \; (II^\Delta)
&<
\delta + \lim_{\epsilon \to 0} \max_{k=1,\dots, N} 2M \int_0^T \int_{L^2_x} \Vert u - u_\epsilon \Vert_{L^2_x} \, 
 \mu^{\Delta_k}_t(du)
\, dt
=
\delta.
\end{align*}
Together with the estimate \eqref{eq:I-est}, the definition of $(I^\Delta)$, $(II^\Delta)$ \eqref{eq:I-II-split}, and \eqref{eq:S-moll-est}, we finally conclude that 
\[
\limsup_{r\to 0} \sup_{\Delta > 0} \S^T_2(\mu^\Delta_t;r) 
< \delta.
\]
But $\delta >0$ was arbitrary, so we must in fact have 
\[
\limsup_{r\to 0} \sup_{\Delta > 0} \S^T_2(\mu^\Delta_t;r) = 0,
\]
as claimed.
\end{proof}
\fi

\bibliographystyle{siam}
\bibliography{main}

\begin{thebibliography}{10}

\bibitem{apte2008data}
{\sc A.~Apte, C.~K. Jones, A.~Stuart, and J.~Voss}, {\em Data assimilation:
  Mathematical and statistical perspectives}, International journal for
  numerical methods in fluids, 56 (2008), pp.~1033--1046.

\bibitem{bansal2021numerical}
{\sc P.~Bansal}, {\em {Numerical approximation of statistical solutions of the
  incompressible Navier-Stokes Equations}}, preprint (arXiv:2107.06073),
  (2021).

\bibitem{bardos2015stability}
{\sc C.~Bardos and E.~Tadmor}, {\em {Stability and spectral convergence of
  Fourier method for nonlinear problems: on the shortcomings of the de-aliasing
  method}}, Numerische Mathematik, 129 (2015), pp.~749--782.

\bibitem{Cho1}
{\sc A.~Chorin.}, {\em Numerical solution of the {Navier-Stokes} equations},
  Math. Comput., 22 (1968), pp.~745--762.

\bibitem{courtier1998ecmwf}
{\sc P.~Courtier, E.~Andersson, W.~Heckley, D.~Vasiljevic, M.~Hamrud,
  A.~Hollingsworth, F.~Rabier, M.~Fisher, and J.~Pailleux}, {\em The {ECMWF}
  implementation of three-dimensional variational assimilation ({3D}-{V}ar).
  {I}: {F}ormulation}, Quarterly Journal of the Royal Meteorological Society,
  124 (1998), pp.~1783--1807.

\bibitem{Ors1}
{\sc Y.~M.~H. D.~Gottlieb and S.~Orszag}, {\em Theory and application of
  spectral methods}, in Spectral methods for PDEs, SIAM, 1984, pp.~1--54.

\bibitem{Dafermos2005}
{\sc C.~M. Dafermos}, {\em Hyperbolic conservation laws in continuum physics},
  vol.~3, Springer, 2005.

\bibitem{evensen2009data}
{\sc G.~Evensen}, {\em Data assimilation: the ensemble {K}alman filter},
  Springer Science \& Business Media, 2009.

\bibitem{FKMT2017}
{\sc U.~S. Fjordholm, R.~K{\"a}ppeli, S.~Mishra, and E.~Tadmor}, {\em
  Construction of approximate entropy measure-valued solutions for hyperbolic
  systems of conservation laws}, Foundations of Computational Mathematics, 17
  (2017), pp.~763--827.

\bibitem{FLMW1}
{\sc U.~S. Fjordholm, K.~Lye, S.~Mishra, and F.~Weber}, {\em Statistical
  solutions of hyperbolic systems of conservation laws: Numerical
  approximation}, Mathematical Models and Methods in Applied Sciences, 30
  (2020), pp.~539--609.

\bibitem{Fjordholm2012}
{\sc U.~S. Fjordholm, S.~Mishra, and E.~Tadmor}, {\em Arbitrarily high-order
  accurate entropy stable essentially nonoscillatory schemes for systems of
  conservation laws}, SIAM Journal on Numerical Analysis, 50 (2012),
  pp.~544--573.

\bibitem{frisch1995turbulence}
{\sc U.~Frisch}, {\em Turbulence: the legacy of {A.N.} {K}olmogorov}, Cambridge
  University Press, 1995.

\bibitem{Ghoshal}
{\sc S.~Ghoshal.}, {\em An analysis of numerical errors in large eddy
  simulations of turbulence}, J. Comput. Phys., 125 (1996), pp.~187--206.

\bibitem{Glimm2001}
{\sc J.~Glimm, J.~Grove, X.~Li, W.~Oh, and D.~Sharp}, {\em {A Critical Analysis
  of Rayleigh-Taylor Growth Rates}}, Journal of Computational Physics, 169
  (2001), pp.~652 -- 677.

\bibitem{GRbook}
{\sc E.~Godlewski and P.-A. Raviart}, {\em Numerical approximation of
  hyperbolic systems of conservation laws}, vol.~118, Springer Science \&
  Business Media, 2013.

\bibitem{GuermondPrudhomme}
{\sc J.-L. Guermond and S.~Prudhomme}, {\em Mathematical analysis of a spectral
  hyperviscosity {LES} model for the simulation of turbulent flows}, ESAIM:
  Mathematical Modelling and Numerical Analysis - Mod\'elisation Math\'ematique
  et Analyse Num\'erique, 37 (2003), pp.~893--908.

\bibitem{Herrmann_2020}
{\sc L.~Herrmann, C.~Schwab, and J.~Zech}, {\em Deep neural network expression
  of posterior expectations in {B}ayesian {PDE} inversion}, Inverse Problems,
  36 (2020), p.~125011.

\bibitem{KaipioSomersalo2006}
{\sc J.~Kaipio and E.~Somersalo}, {\em Statistical and computational inverse
  problems}, vol.~160, Springer Science \& Business Media, 2006.

\bibitem{Karamanos2000a}
{\sc G.~S. Karamanos and G.~E. Karniadakis}, {\em A spectral vanishing
  viscosity method for large-eddy simulations}, J. Comput. Phys., 163 (2000),
  pp.~22--50.

\bibitem{FNObounds}
{\sc N.~Kovachki, S.~Lanthaler, and S.~Mishra}, {\em {On Universal
  Approximation and Error Bounds for Fourier Neural Operators}}, Journal of
  Machine Learning Research, 22 (2021), pp.~1--76.

\bibitem{Ladyzhenskaya}
{\sc O.~A. Ladyzhenskaya}, {\em The mathematical theory of viscous
  incompressible flow}, vol.~2 of Mathematics and its applications, New York:
  Gordon and Breach, 1969.

\bibitem{LM2019}
{\sc S.~Lanthaler and S.~Mishra}, {\em {On the convergence of the spectral
  viscosity method for the two-dimensional incompressible Euler equations with
  rough initial data}}, Foundations of Computational Mathematics,  (2019),
  pp.~1--54.

\bibitem{LMP2}
{\sc S.~Lanthaler, S.~Mishra, and C.~Parés-Pulido}, {\em On the conservation
  of energy in two-dimensional incompressible flows}, Nonlinearity, 34 (2021),
  pp.~1084--1135.
\newblock Publisher: IOP Publishing.

\bibitem{LMP1}
\leavevmode\vrule height 2pt depth -1.6pt width 23pt, {\em Statistical
  solutions of the incompressible {Euler} equations}, Mathematical Models and
  Methods in Applied Sciences, 31 (2021), pp.~223--292.
\newblock \_eprint: https://doi.org/10.1142/S0218202521500068.

\bibitem{Latz2020}
{\sc J.~Latz}, {\em {On the Well-posedness of Bayesian Inverse Problems}},
  SIAM/ASA Journal on Uncertainty Quantification, 8 (2020), pp.~451--482.

\bibitem{law2015data}
{\sc K.~Law, A.~Stuart, and K.~Zygalakis}, {\em Data assimilation}, Cham,
  Switzerland: Springer, 214 (2015).

\bibitem{EvaluatingDataAssimilationAlgorithms}
{\sc K.~J.~H. Law and A.~M. Stuart}, {\em Evaluating data assimilation
  algorithms}, Monthly Weather Review, 140 (2012), pp.~3757 -- 3782.

\bibitem{Leray1934}
{\sc J.~Leray}, {\em Sur le mouvement d'un liquide visqueux emplissant
  l'espace}, Acta mathematica, 63 (1934), pp.~193--248.

\bibitem{Leveque}
{\sc R.~J. LeVeque}, {\em Numerical methods for conservation laws}, vol.~3,
  Springer, 1992.

\bibitem{fourierop2020}
{\sc Z.~Li, N.~B. Kovachki, K.~Azizzadenesheli, B.~Liu, K.~Bhattacharya,
  A.~Stuart, and A.~Anandkumar}, {\em Fourier neural operator for parametric
  partial differential equations}, in International Conference on Learning
  Representations, 2021.

\bibitem{mishra2021wellposedness}
{\sc S.~Mishra, D.~Ochsner, A.~M. Ruf, and F.~Weber}, {\em Well-posedness of
  {B}ayesian inverse problems for hyperbolic conservation laws}, preprint
  (arXiv:2107.09701),  (2021).

\bibitem{pope2001turbulent}
{\sc S.~B. Pope}, {\em Turbulent flows}, Cambridge University Press, 2001.

\bibitem{rabier2000ecmwf}
{\sc F.~Rabier, H.~J{\"a}rvinen, E.~Klinker, J.-F. Mahfouf, and A.~Simmons},
  {\em {The ECMWF operational implementation of four-dimensional variational
  assimilation. I: Experimental results with simplified physics}}, Quarterly
  Journal of the Royal Meteorological Society, 126 (2000), pp.~1143--1170.

\bibitem{tapio}
{\sc T.~Schneider, S.~Lan, A.~Stuart, and J.~Teixeira}, {\em Earth
  systemmodeling $2.0$: A blueprint for models that learn from observations and
  targeted high-resolution simulations}, Geophysical Research Letters, 44
  (2017), pp.~12396--12417.

\bibitem{Sprungk2020}
{\sc B.~Sprungk}, {\em {On the local Lipschitz stability of Bayesian inverse
  problems}}, Inverse Problems, 36 (2020), p.~055015.

\bibitem{Stuart2010}
{\sc A.~M. Stuart}, {\em {Inverse problems: a Bayesian perspective}}, Acta
  numerica, 19 (2010), pp.~451--559.

\bibitem{Tadmor1989}
{\sc E.~Tadmor}, {\em Convergence of spectral methods for nonlinear
  conservation laws}, SIAM J. Numer. Anal., 26 (1989).

\bibitem{Tadmor2004}
\leavevmode\vrule height 2pt depth -1.6pt width 23pt, {\em Burgers' equation
  with vanishing hyper-viscosity}, Communications in Mathematical Sciences, 2
  (2004), pp.~317--324.

\bibitem{Tarantola2005}
{\sc A.~Tarantola}, {\em Inverse problem theory and methods for model parameter
  estimation}, Society for Industrial and Applied Mathematics (SIAM),
  Philadelphia, PA, 2005.

\end{thebibliography}

\end{document}